\documentclass{siamltex}

\usepackage{fixltx2e} 
\usepackage{cmap} 
\usepackage{graphicx}
\usepackage{amssymb}
\usepackage{amsmath}
\usepackage[numbers]{natbib}
\usepackage{subfig}

\newtheorem{thm}{Theorem}

\newtheorem{lem}[thm]{Lemma}

\newtheorem{rem}[thm]{Remark}

\newcommand{\R}{\mathbb{R}}
\newcommand{\ud}{u^\dagger} 
\newcommand{\udd}{u^\ddagger} 
\newcommand{\sd}{s^\dagger}
\renewcommand{\matrix}[2]{ \left(\begin{array}{#1} #2 \end{array}\right)}

\title{Phase Resetting in an Asymptotically Phaseless System:  On the Phase Response of Limit Cycles Verging on a Heteroclinic Orbit.}
\author{\normalsize\textsc{%
    Kendrick M. Shaw$^1$, %
    Hillel J. Chiel$^{1,2,3}$, %
    Peter J. Thomas$^{1,4,5}$}\\
\small Departments of $^1$Biology, $^2$Biomedical Engineering, $^3$Neurosciences,\\ 
\small $^4$Mathematics and $^5$Cognitive Science\\
\small Case Western Reserve University, 10900 Euclid Avenue, Cleveland Ohio 44106}

\newlength{\figwidth}
\figwidth=\linewidth

\newcommand{\BMH}{\cite{BrownMoehlisHolmes:2004:NeComp} }

\begin{document}
\maketitle
\begin{abstract}
Rhythmic behaviors in neural systems often combine features of limit cycle dynamics (stability and periodicity) with features of near heteroclinic or near homoclinic cycle dynamics (extended dwell times in localized regions of phase space).  Proximity of a limit cycle to one or more saddle equilibria can have a profound effect on the timing of trajectory components and response to both fast and slow perturbations, providing a possible mechanism for adaptive control of rhythmic motions. Reyn showed that for a planar dynamical system with a stable heteroclinic cycle (or separatrix polygon), small perturbations satisfying a net inflow condition will generically give rise to a stable limit cycle (Reyn, 1980; Guckenheimer and Holmes, 1983).  Here we consider the asymptotic behavior of the infinitesimal phase response curve (iPRC) for examples of two systems satisfying Reyn's inflow criterion, (i) a smooth system with a chain of  four hyperbolic saddle points and (ii) a piecewise linear system corresponding to local linearization of the smooth system about its saddle points.  For system (ii), we obtain exact expressions for the limit cycle and the iPRC as a function of a parameter $\mu>0$ representing the distance from a heteroclinic bifurcation point.  In the $\mu\to 0$ limit, we find that perturbations parallel to the unstable eigenvector direction in a piecewise linear region lead to divergent phase response, as previously observed (Brown, Moehlis and Holmes (2004) \textit{Neural Computation}).  In contrast to previous work, we find that perturbations parallel to the \emph{stable} eigenvector direction can lead to either divergent or convergent phase response, depending on the phase at which the perturbation occurs.  In the smooth system (i), we show numerical evidence of qualitatively similar phase specific sensitivity to perturbation.
\end{abstract}

keywords:  (MSC database) 
\begin{itemize}
\item 	34C05  	Location of integral curves, singular points, limit cycles
\item 	37C27  	Periodic orbits of vector fields and flows
\item		37C29  	Homoclinic and heteroclinic orbits
\item 	37G15  	Bifurcations of limit cycles and periodic orbits
\item 	70K05  	Phase plane analysis, limit cycles
\item		70K44  	Homoclinic and heteroclinic trajectories
\end{itemize}

Movies:
\begin{itemize}
\item \begin{verbatim}http://dynamicspjt.case.edu/~kms15/smooth.mpg\end{verbatim}
\item \begin{verbatim}http://dynamicspjt.case.edu/~kms15/iris.mpg\end{verbatim}
\item \begin{verbatim}http://dynamicspjt.case.edu/~kms15/iris_isochrons.mpg\end{verbatim}
\end{itemize}
\section{Introduction}

Animals often generate specific 
sequences of behavior, for example the movements of the limbs during walking,
the feeding apparatus while chewing and swallowing, or body undulations in
swimming.  When a repeated sequence of motions can be produced reliably, the
pattern generator circuit controlling the behavior is typically modeled as an autonomous system of ordinary differential equations admitting a stable isolated periodic orbit, \textit{i.e.}~a 
limit-cycle oscillator \cite{Ijspeert:2008:NeuralNet,Wilson1999SpikesBook}. 
Limit cycle oscillators have played a fundamental role in understanding the generation and control of repetitive motions underlying swimming in the lamprey \cite{BuchananCohen:1982:JNeurophys,CohenErmentroutKiemelKopellSigvardtWilliams:1992:TrendsNsci}, neural activity and bursting \cite{CoombesBressloff2005BurstingBook,Izhikevich2007}, the gaits of quadrupeds  \cite{BuonoGolubitsky:2001:JMathBiol,GolubitskyStewartBuonoCollins:1999:Nature} bipeds \cite{PintoGolubitsky:2006:JMathBiol} and monopeds \cite{BeerChielGallagher:1999:JComputNsci,ChielBeerGallagher:1999:JComputNsci}, as well as cardiac and respiratory activity \cite{De-SchutterAngstadtCalabrese:1993:JNPhys,RubinShevtsovaErmentroutSmithRybak:2009:JNeurophys,Sammon:1994:JApplPhysiol}, \textit{inter alia}.  


Another class of systems generating reproducible sequences of activity has been proposed under the rubric of \emph{stable heteroclinic sequences} \cite{AfraimovichZhigulinRabinovich:2004:Chaos} or \emph{stable heteroclinic channels} \cite{RabinovichEtAl2008PLoS-CB}.  A dynamical system possesses a heteroclinic sequence if there exists a chain of hyperbolic saddle fixed points for which the unstable manifold of each saddle intersects the stable manifold of the next \cite{AfraimovichTristanHuertaRabinovich:2008:Chaos}; they generalize the notion of a stable heteroclinic cycle, an attractor comprising a finite collection of saddle points with heteroclinic connections linking them in a repeating chain \cite{ArmbrusterStoneKirk:2003:Chaos,Guckenheimer+Holmes1990,StoneArmbruster:1999:Chaos,StoneHolmes1990SIAMJApplMath}.
A heteroclinic cycle is \emph{stable} or \emph{attracting} when the product around the cycle of the saddle values is strictly greater than unity; the saddle value for a hyperbolic saddle point with eigenvalues $\lambda_u>0>\lambda_s^1\ge\lambda_s^2\ge\cdots\ge\lambda_s^{n-1}$ is the ratio $-\lambda_s^1/\lambda_u$.  Stable heteroclinic sequences and cycles have been proposed to provide a framework within nonlinear dynamics for understanding a range of phenomena, including olfactory processing in insects \cite{RabinovichHuertaLaurent:2008:Science}, search behavior in the marine mollusk \textit{Clione limacina} \cite{VaronaRabinovichSelverstonArshavsky:2002:Chaos}, ``winnerless competition'' in neural circuits \cite{AfraimovichOrdazUrias2002Chaos,AfraimovichRabinovichVarona2004IntJBifChaos} and ecological models \cite{AfraimovichTristanHuertaRabinovich:2008:Chaos}, genesis of network-dependent bursting activity \cite{NowotnyRabinovitch2007PRL}, and the balance of emotion and cognition in behavioral control \cite{AfraimovichYoungMuezzinogluMehmetRabinovich:2010:BullMathBiol} as well as other areas \cite{RabinovichEtAl2006RMP}.

Because they require the intersection of one dimensional unstable and codimension one stable manifolds, the saddle connections comprising a heteroclinic cycle are structurally unstable \cite{Andronov+Pontryagin:1937}.  For planar systems, Reyn showed that a phase portrait containing a separatrix polygon generically forms a limit cycle when subject to a perturbation satisfying a net inflow condition \cite{Reyn1980-bookchapter}, provided the unperturbed heteroclinic cycle is attracting.  Trajectories near an unperturbed attracting heteroclinic cycle traverse the cycle with longer and longer return times; the trajectory along the cycle itself has ``infinite period".  

 As an example,
consider the system defined on the 2-torus $0\le y_i\le 2\pi$, $i\in\{1,2\}$.  
\begin{subequations}
\label{eq:smooth_sho}
\begin{align}
\frac{dy_1}{dt}&= f(y_1,y_2)=\cos(y_1) \sin(y_2) + \alpha \sin(2 y_1) \\
\frac{dy_2}{dt}&= g(y_1,y_2)=-\sin(y_1) \cos(y_2) + \alpha \sin(2 y_2) 
\end{align}
\end{subequations}
As shown in Figure~\ref{fig:sine_flow_shc}, this system has four saddle points
connected by heteroclinic connections.  The saddles have identical eigenvalues
$\lambda_u= 1 - 2\alpha, \lambda_s= -1 - 2\alpha$ for a saddle value of 
$V=\left(\frac{1+2\alpha}{1-2\alpha}\right)^4$.  If $\alpha \in (0, 1/2)$ then 
$V > 1$,
indicating that the heteroclinic 4-cycle is attracting.  

\begin{figure} \centering
\subfloat[]{\includegraphics[width=.45\figwidth]{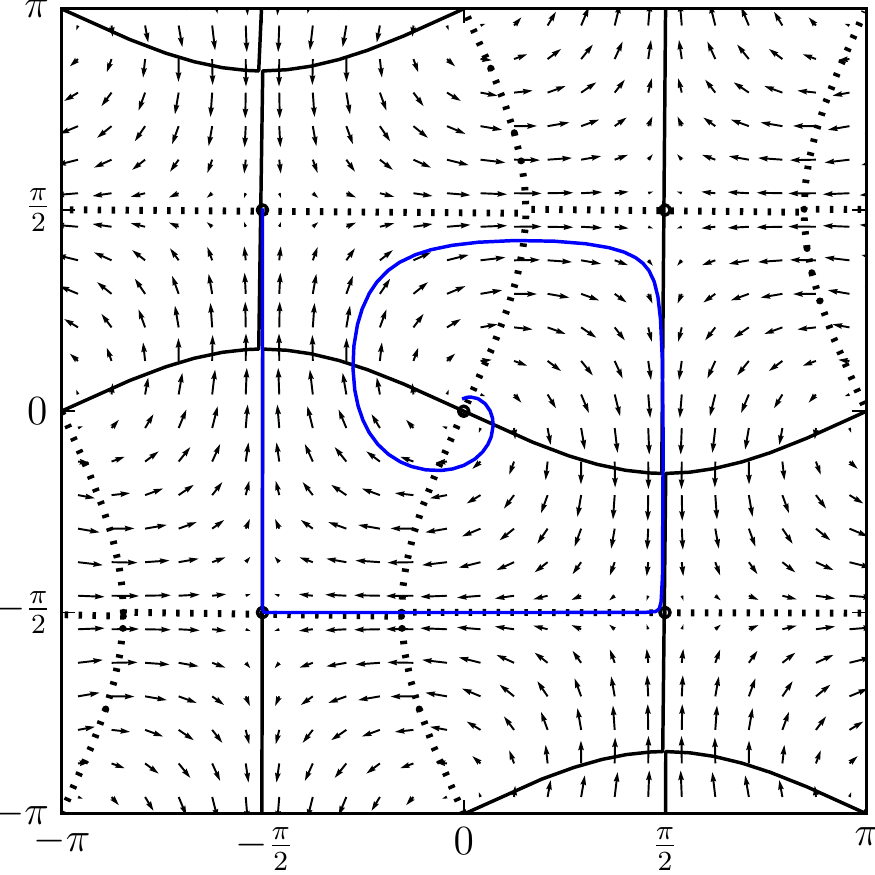}
    \label{fig:sine_flow_shc}}
\hskip 0.5cm
\subfloat[]{\includegraphics[width=.45\figwidth]{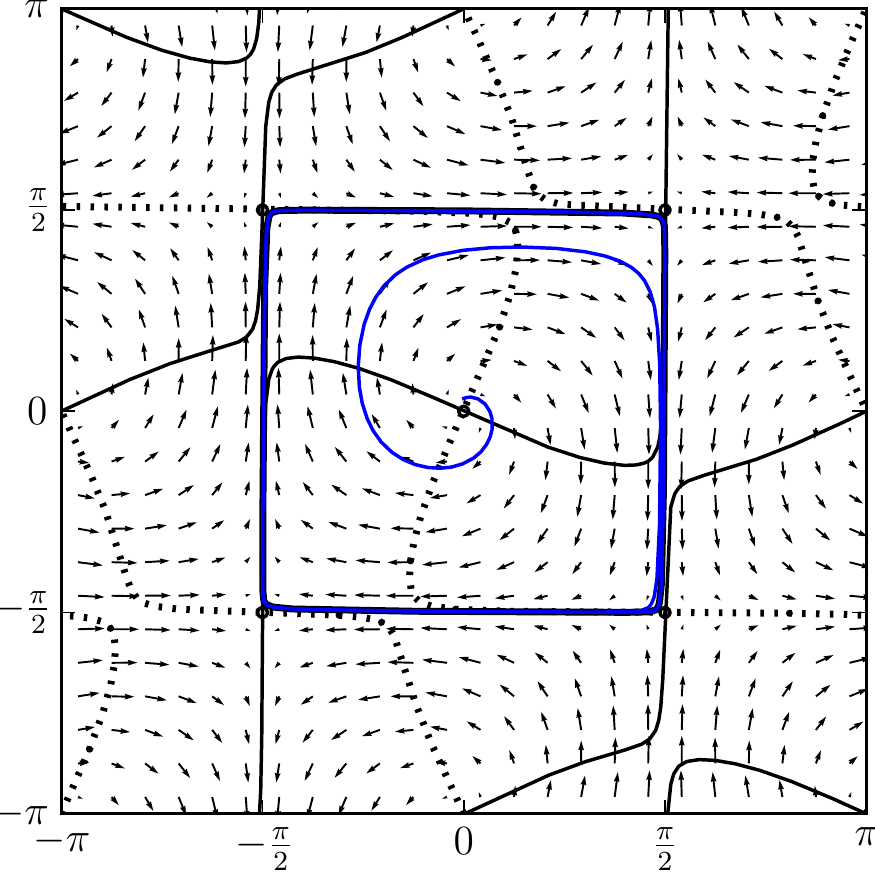}
    \label{fig:sine_flow_lc}}
\caption{Parametric perturbation of an attracting heteroclinic cycle to a stable limit cycle.  Note solid line is x-nullcline and dashed is y-nullcline.
\subref{fig:sine_flow_shc}  A trajectory of the smooth toroidal system given by Equations \ref{eq:smooth_sho}, passing near four distinct saddles each with eigenvalues 
$\lambda_s=-1 - 2\alpha$ and $\lambda_u=1 - 2\alpha$.
The trajectory shown begins near the unstable spiral point at location $(\pi,\pi)$ near the center
of the plot.  The trajectory was integrated for a total time of 200 units  
using $\alpha = 7/30$ and $\mu=0$
.  It passes closer to each successive saddle, slowing progressively. 
\subref{fig:sine_flow_lc} The perturbed system, Equations \ref{eq:rotated_smooth_sho}, when
$\alpha = 7/30$ and
$\mu=0.01$, forms a stable limit cycle passing close to the four saddle points.  Please see corresponding
movie (file \texttt{smooth.mpg}); this animation cycles through phase portraits for the smooth system for values of $\mu$ varying from $0$ to $1/2$ and back again.  At $\mu=2 \alpha$ the limit cycle collapses \textit{via} a Hopf bifurcation to a single stable fixed point located between the four surrounding saddles.}
\label{fig:shc-intro}
\end{figure}


Although the heteroclinic cycle connecting the four saddle points 
is structurally unstable, any perturbation of the vector field that pushes the unstable manifold of each saddle towards the interior of the cycle relative to the stable manifold of the next saddle will generically lead to the formation of a stable limit cycle.
For example, consider the effects of ``rotation'' of the flow of 
Equations~\ref{eq:smooth_sho} parametrized by  $1>\mu > 0$.
\begin{subequations}
\label{eq:rotated_smooth_sho}
\begin{align}
\frac{dy_1}{dt}&= f(y_1, y_2) + \mu g(y_1, y_2) \\
\frac{dy_2}{dt}&= g(y_1, y_2) - \mu f(y_1, y_2) 
\end{align}
\end{subequations}
We will refer to this system as the \emph{smooth system}, in contrast to the piecewise linear system to be introduced subsequently.  As shown in
Figure~\ref{fig:sine_flow_lc}, for $\mu>0$ the heteroclinic connections are broken,
but a limit cycle has been created inside of the saddles that passes near each
of the saddle points.  

For any small positive value of $\mu$, the system \ref{eq:rotated_smooth_sho} will have a stable limit cycle, $\gamma(t)=\gamma(t+T),$ with period $T(\mu)\to\infty$ as $\mu\to 0^+$.  For $\mu>0$ we may
identify a phase $\theta \in[0,\theta_{\mbox{max}})$ with each point on the limit cycle by chosing an arbitrary point $\gamma_0$ to have phase $\theta(\gamma_0)=0$ and requiring $d\theta/dt=2\pi/T(\mu)$. 
Typically one chooses $\theta_{\mbox{max}}$ to equal either $2\pi$ or unity.  Because of the fourfold symmetry of the systems considered here it will be convenient to set $\theta_{\mbox{max}}=4$ throughout, so that traversal of each quarter of the limit cycle will correspond to a unit increment in phase.  Figure \ref{fig:sine_timeplot} illustrates the time course of trajectories of Equations \ref{eq:rotated_smooth_sho} for $\alpha=7/30$ and $\mu\in\{10^{-3},0.1,0.3,0.45\}$.

\begin{figure}[htpb] \centering
\includegraphics[width=\figwidth]{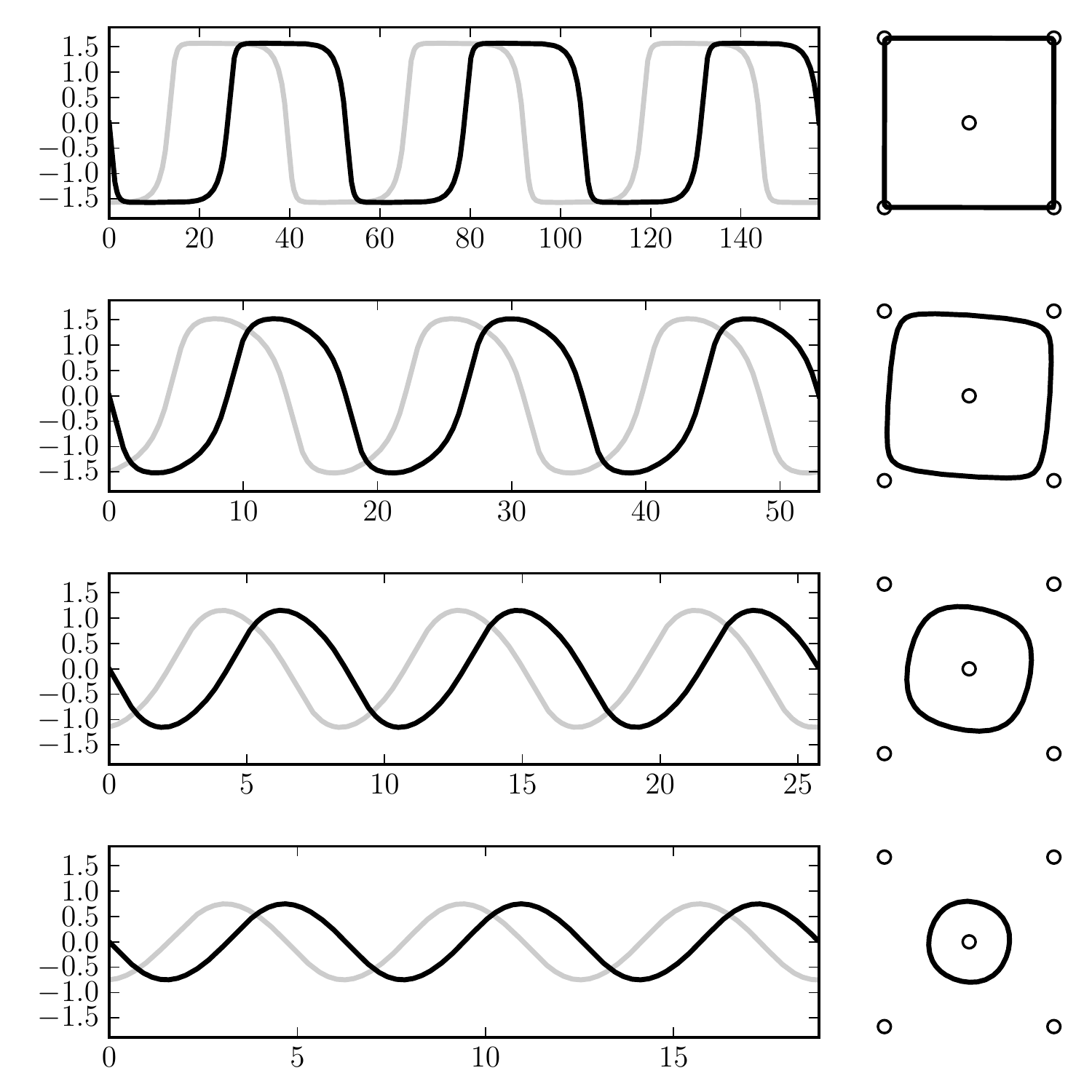}
\caption{
Time plots of limit cycle trajectories of the smooth system with 
various values of $\mu$.  When $\mu \ll 1$ the time plots exhibit prolonged dwell times  due to slow transits past the saddle points and relatively rapid transits between
them.  As $\mu$ increases, the limit cycle moves away from the saddle
points and the speed becomes more uniform (note the change in horizontal
scale).  As a result, the time plots become more sinusoidal resembling an
Andronov--Hopf oscillator.  The first and second ordinates are shown in black 
and gray, respectively, with $\alpha = 7/30$.  From top to bottom 
$\mu= 10^{-3}, 0.1, 0.3, 0.45$. Compare Figure \ref{fig:iris_timeplot}.}
\label{fig:sine_timeplot}
\end{figure}

To each point $\xi_0=(x_0,y_0)$ in the basin of attraction of the limit cycle we may assign an asymptotic phase $\phi(\xi_0)\in[0,4)$ satisfying 
$||\xi(t)-\gamma(t-\theta(\xi_0)T/4)||\to 0$ 
as $t\to\infty$, where $\xi(t)$ and $\gamma(t)$ are solutions of \ref{eq:rotated_smooth_sho} satisfying initial
conditions  $\xi(0)=\xi_0$ and $\gamma(0)=\gamma_0$ respectively, and $||(x,y)||=|x|+|y|$.
The level curves of $\phi$, called isochrons, foliate the basin of attraction; see \cite{Izhikevich2007} for examples.

An important question for understanding the dynamics of control in central pattern generators is how such systems combine robustness to perturbation with flexibility to adapt behavior to variable environmental or physiological conditions.\footnote{Whether originating endogenously (neural noise, internal control mechanisms), or exogenously (environmental fluctuations), perturbations of the system's dynamics occur on a variety of time scales.  For clarity we will limit discussion to two extremes of fast and slow perturbations.  Perturbations occurring on slow time scales relative to other system dynamics will be referred to below as \emph{static} or \emph{parametric} perturbations, for example fixing a value $\mu>0$ in Equations \ref{eq:rotated_smooth_sho}.  Fast perturbations will be approximated as instantaneous trajectory dislocations.}  
Many biological systems combining repetitive behavior and behavioral or metabolic control exhibit limit cycle behavior that is strongly influenced by passage of trajectories near one or more unstable fixed points or quasiequilibria.  
A recent model for the generation of multiple rhythmic states in a respiratory CPG featured slow excursions along equilibrium surfaces of model neuronsÕ voltage dynamics interspersed with fast jumps between these
surfaces \cite{RubinShevtsovaErmentroutSmithRybak:2009:JNeurophys}.  In this model, control of the rhythm through changes in the period as well as the duration of different functional phases (inspiratory phase, expiratory phase) could be effected by modulatory signals making small changes to the dynamics of escape and release from inhibition, thereby changing the paths of trajectories in the vicinity of quasiequilibria.   Similarly, control of the net speed of motion produced by a model locomotory CPG coupled to an explicit musculoskeletal system resulted from the adjustment of CPG trajectories in proximity to unstable fixed points of the model \cite{SpardyEtAlRubin2010SFN-conf}.  Patterns of biting, swallowing and rejection in \textit{Aplysia} may be understood in terms of sequential traversals between neuromechanical equilibrium points \cite{sutton+etal+chiel:2004:biolcyb,SuttonEtAlChiel:2004:JCompPhys}.
During normal cell growth and proliferation a living cell passes repeatedly through several phases (including cell division), yet the cell ``cycle"  is typically described not as a standard limit
cycle but as a sequence of traversals between quasiequilibria that act as checkpoints \cite{TysonNovak:2001:JTB}.\footnote{Strictly speaking, the system may be viewed as a limit cycle if the dynamics are embedded in a larger space encompassing the control variables as well; the point remains that the periodic behavior is strongly influenced by passage near unstable equilibria or near-equilibria, which function to regulate the timing of the system.}  
In the vicinity of each quasiequilibrium, the cell cycle dynamics slows for an indefinite period of time, until a regulating condition is met \cite{NovakTyson:2008:NatRevMolCellBiol}. 
In each of these examples, although the flow strictly speaking forms a deterministic limit cycle, the behavior may be actively managed through the introduction of \emph{variable dwell times} that function as control points along trajectories.

Families of limit cycles verging on a heteroclinic (or homoclinic) cycle such as the limit cycles of the one parameter family of systems given by Equations \ref{eq:rotated_smooth_sho} provide an opportunity for studying the role of saddle points in the control of timing of rhythmic behaviors.  As a first step, it is natural to consider the structure of the infinitesimal phase resetting curves which reflect the sensitivity of the return time along the limit cycle to small instantaneous perturbations.  When $\mu=0$ the flow of Equations \ref{eq:rotated_smooth_sho} does not admit a periodic solution with finite period; consequently the asymptotic phase and hence the phase response is not well defined.  However, for any $\mu>0$ we can defined the phase response, and we can study its behavior as the family of limit cycles approaches the heteroclinic cycle.  

To fix terminology, for $1\gg\mu>0$, consider a trajectory $\xi(t)$ following a stable limit cycle $\gamma:t\in[0,T)\to \gamma(t)\in\R^n$.  Suppose the trajectory is perturbed by a small instantaneous displacement, taking $x(t^-)=\gamma(t+\theta_0T/4)$  to $x(t^+)=\gamma(t+\theta_0T/4)+\vec{r}$. 
Provided the new initial condition $x(t^+)$ remains within the basin of
attraction of the limit cycle, the orbit will approach the limit cycle with
$||x(t)-\gamma((t+\theta_1T/4)\mod T)||\to 0$ as $t\to\infty$, for some new phase
$\theta_1\in[0,4)$, resulting in a shift in asymptotic phase equal to
$(\theta_1-\theta_0)\mod 4$.  The limit cycle's sensitivity to weak
instantaneous perturbations typically varies both with the phase at which the perturbation occurs, $\theta_0(t)=4((t/T)\mod 1)$,
 and with the direction in which it occurs,
$\eta=\vec{r}/||\vec{r}||\in\R^n$.  Differences
in sensitivity at different phases, which are important for
understanding possible control mechanisms, are captured by 
the  \emph{infinitesimal Phase Response Curve} or \emph{iPRC}, defined as 
\begin{equation}\label{eq:PRC-intro}
Z(\theta,\mu,\eta)=\lim_{\epsilon\to 0} \frac{(\theta-\phi(\gamma(\theta T/4)+\epsilon\eta))\mod 4}{\epsilon}.
\end{equation}
As above, $\phi(\xi)$ is the asymptotic phase associated with a point $\xi$ in the basin of attraction for the limit cycle, $\theta$ is the phase at which the instantaneous perturbation is applied, $\mu$ is the parameter controlling proximity to the heteroclinic, $\eta$ is the direction of the fast perturbation, $\gamma$ is the limit cycle trajectory and $T$ is its period; the latter two entities are functions of $\mu$.
In order to study the behavior of the phase response as the system approaches the heteroclinic configuration, it is important to emphasize the limit defining the iPRC in Equation \ref{eq:PRC-intro} is taken \emph{before} the subsequent limit $\mu\to 0$.

Analytic calculation of the iPRC is typically accomplished \textit{via} an adjoint equation method, and exact solutions are known in very few cases \cite{ErmentroutTerman2010book}.  In order to perform the analysis required to gain qualitative insight into the behavior of the phase response under the sequential limits $\epsilon\to 0$, $\mu\to 0$, we construct a piecewise linear approximation to the smooth system.  The \emph{iris system}, described below, is topologically equivalent on an open set including the family of limit cycles and the saddle points, and qualitatively captures the behavior of the smooth system.  For the iris system we obtain exact results including the form of the limit cycle for positive $\mu$ and an explicit formula for the infinitesimal phase response curve.  Our main result, stated in Theorem \ref{thm:main} below, shows that for the iris system the sensitivity to small displacements \emph{parallel to the direction of the stable manifold} has two regions with distinct sensitivity behavior.  In the limit, as the family of orbits approaches the SHC, there is an interval of phases $\varphi\in[0,\varphi_c]$ for which the infinitesimal phase response goes to zero.  This interval is separated by a critical phase $\varphi_c$ from a second interval $\varphi\in(\varphi_c,1]$ for which the iPRC diverges to $+\infty$.  The junction $\varphi=0\mod 1$ corresponds to the boundary between regions on which we define a piecewise linear dynamics.  We show numerically that qualitatively similar results hold for the smooth system.


Limit cycles in piecewise linear (PWL) dynamical system have been studied
previously in several contexts. For instance, in the context of Glass networks \cite{Edwards+Glass:2000:Chaos,GlassPasternack1978BMB,GlassPasternack1978JMB}
PWL systems have been used to represent the dynamics of idealized genetic regulatory systems.  
In this case, the
structure is somewhat different from that considered here, in that the fixed point driving the
flow in each piecewise linear region is strictly attracting and lies
outside the region, rather than having one unstable eigendirection and lying inside the corresponding flow region.  Consequently, families of limit cycles verging on a
heteroclinic cycle do not appear in Glass networks.  PWL planar dynamical systems in which a fixed point and a limit cycle coexist do occur in models approximating the Fitzhugh-Nagumo equations; such systems were used to study traveling wave phenomena \cite{McKean1970AdvMath} and period adding bifurcations under periodic forcing \cite{CoombesOsbaldestin:2000:PRE}.   Neither homoclinic nor heteroclinic cycles appear in these systems, however.  
Recently,  Coombes investigated phase response curves for limit cycles in both the PWL McKean-Nagumo model and a new model related to the Morris-Lecar system \cite{Coombes:2008:SIADS};  this paper exploited the existence of exact solutions for the PRC to study synchronization in gap-junction coupled networks, in both the strong and weak coupling limits.  The PWL Morris-Lecar system does contain a homoclinic bifurcation;
to the best of our knowledge, however, the analysis presented here is the first to obtain exact results for the scaling of the infinitesimal phase response curve, as a system of limit cycles approaches a heteroclinic or homoclinic orbit.

\section{The piecewise linear iris system}
\label{ssec:iris-definition}

\begin{figure}[htbp] \centering
\subfloat[]{\includegraphics[width=.4\figwidth]{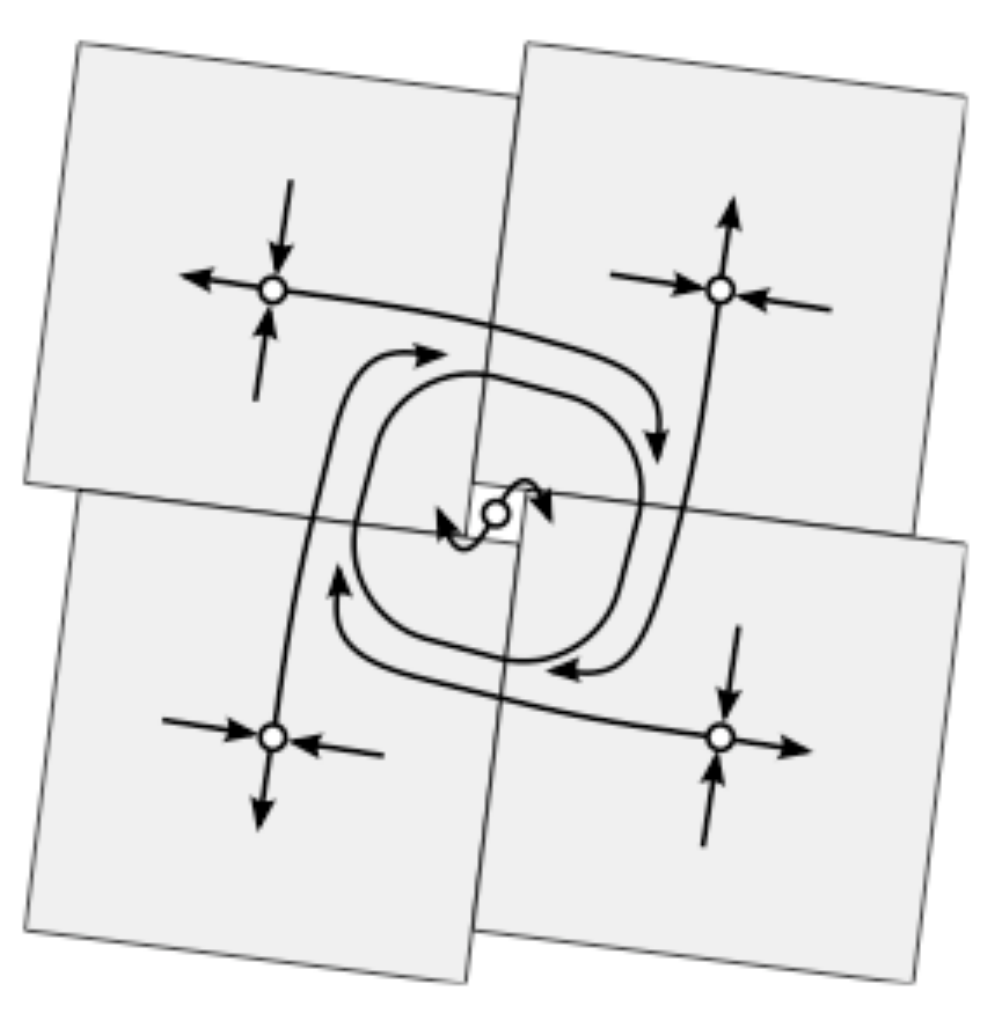}
    \label{fig:sine_to_iris_fig_sine}}
\hskip 0.5cm
\subfloat[]{\includegraphics[height=.4\figwidth]{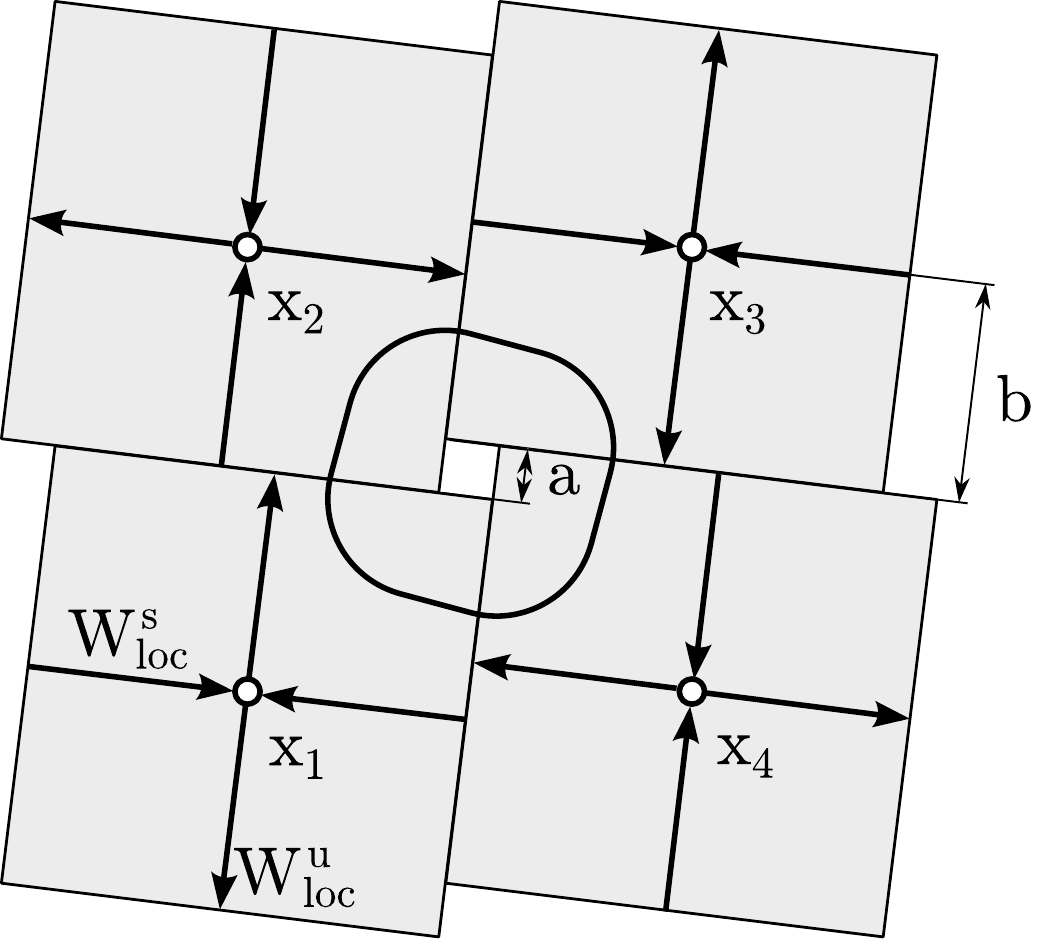}
    \label{fig:sine_to_iris_fig_iris}}
\caption{Construction of the iris system from the smooth system.  
\subref{fig:sine_to_iris_fig_sine} The flow near each saddle point of the smooth system 
(\ref{eq:rotated_smooth_sho}) is approximately linear in a square region aligned with the
(orthogonal) eigenvectors.  By virtue of the fourfold symmetry of the system we can extend
each square until it has a side of length $2b$, remaining centered on the saddle.  Depending on the 
extent of the rotational (parametric) perturbation of the vector field, there will be square of side $a$ forming a gap
around the unstable spiral point.  
\subref{fig:sine_to_iris_fig_iris} We extend the linearized flow for each saddle throughout the corresponding square of side $2b$, creating a piecewise linear vector field defined on the plane with the squares of side $a$ removed.  
The $i^{th}$ square is centered on the $i^{th}$ saddle, $x_i$ ($i=1,2,3,4$).  In each square
the inward arrows indicate $W^s_{loc}$ and the outward arrows indicate $W^u_{loc}$, parallel
to the stable and unstable eigenvector directions, respectively.
  }
\label{fig:sine_to_iris}
\end{figure}
\begin{figure}[htbp] \centering
\subfloat[]{\includegraphics[width=.4\figwidth]{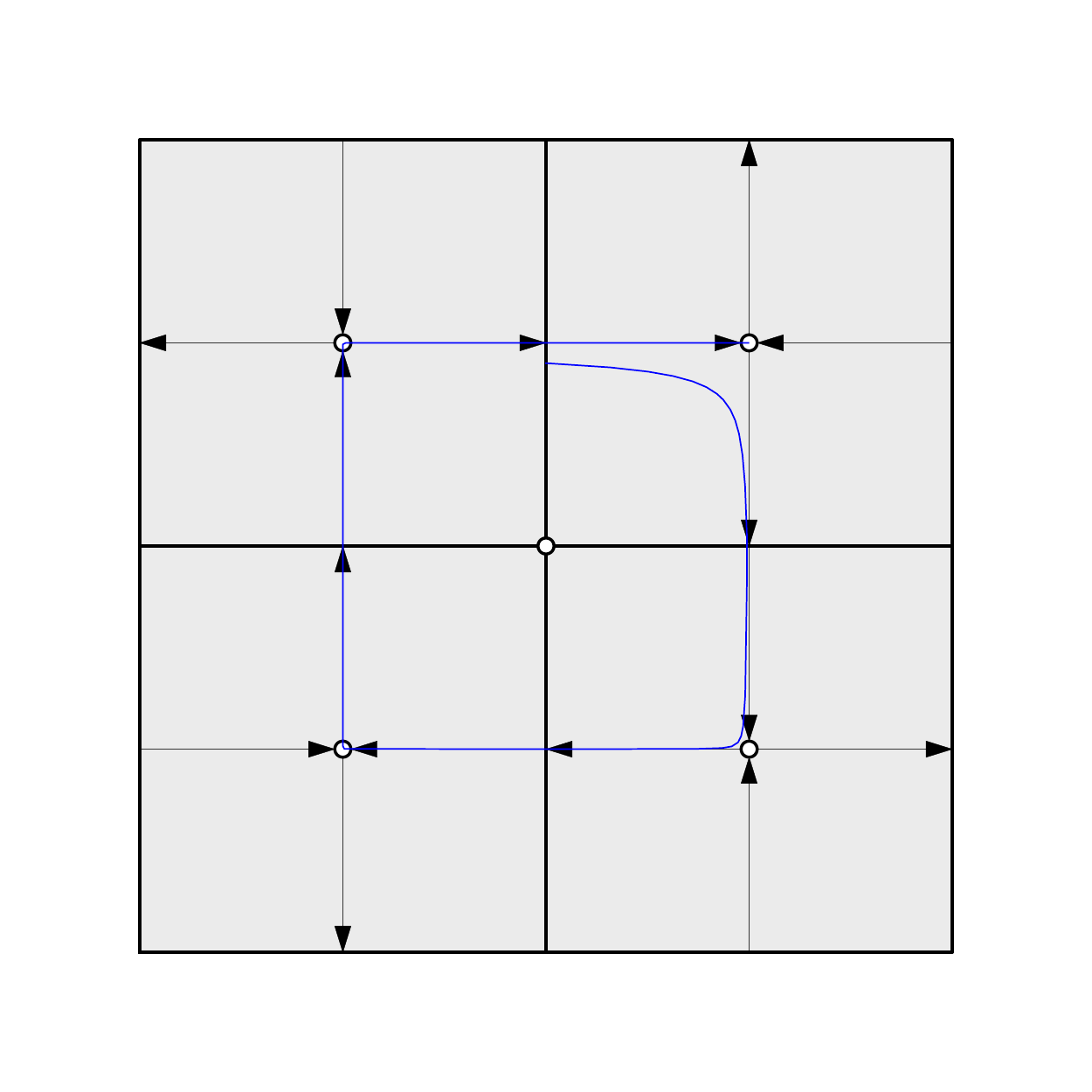}\label{iris_shc}}
\subfloat[]{\includegraphics[width=.4\figwidth]{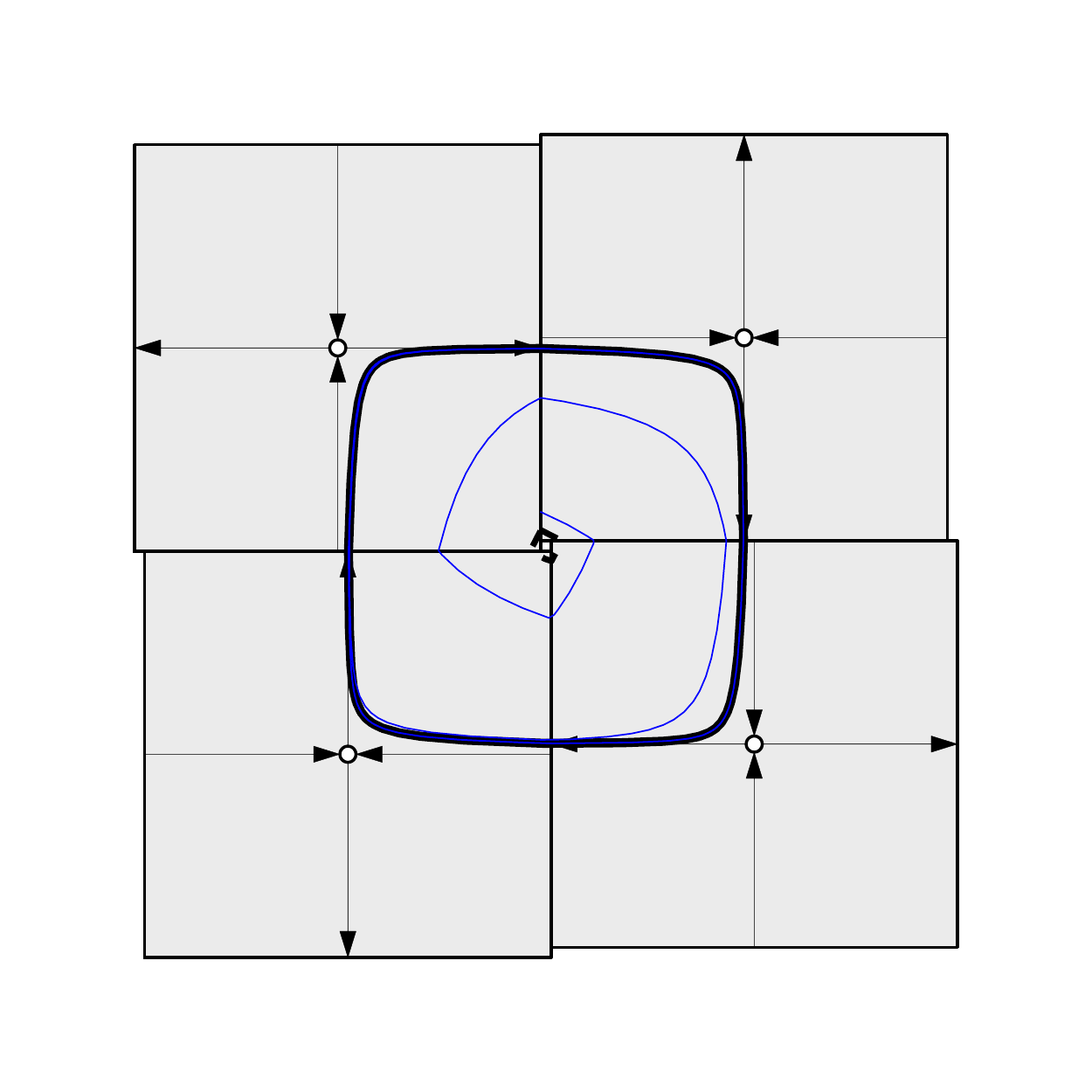}\label{iris_lc0}}\\
\noindent
\subfloat[]{\includegraphics[width=.4\figwidth]{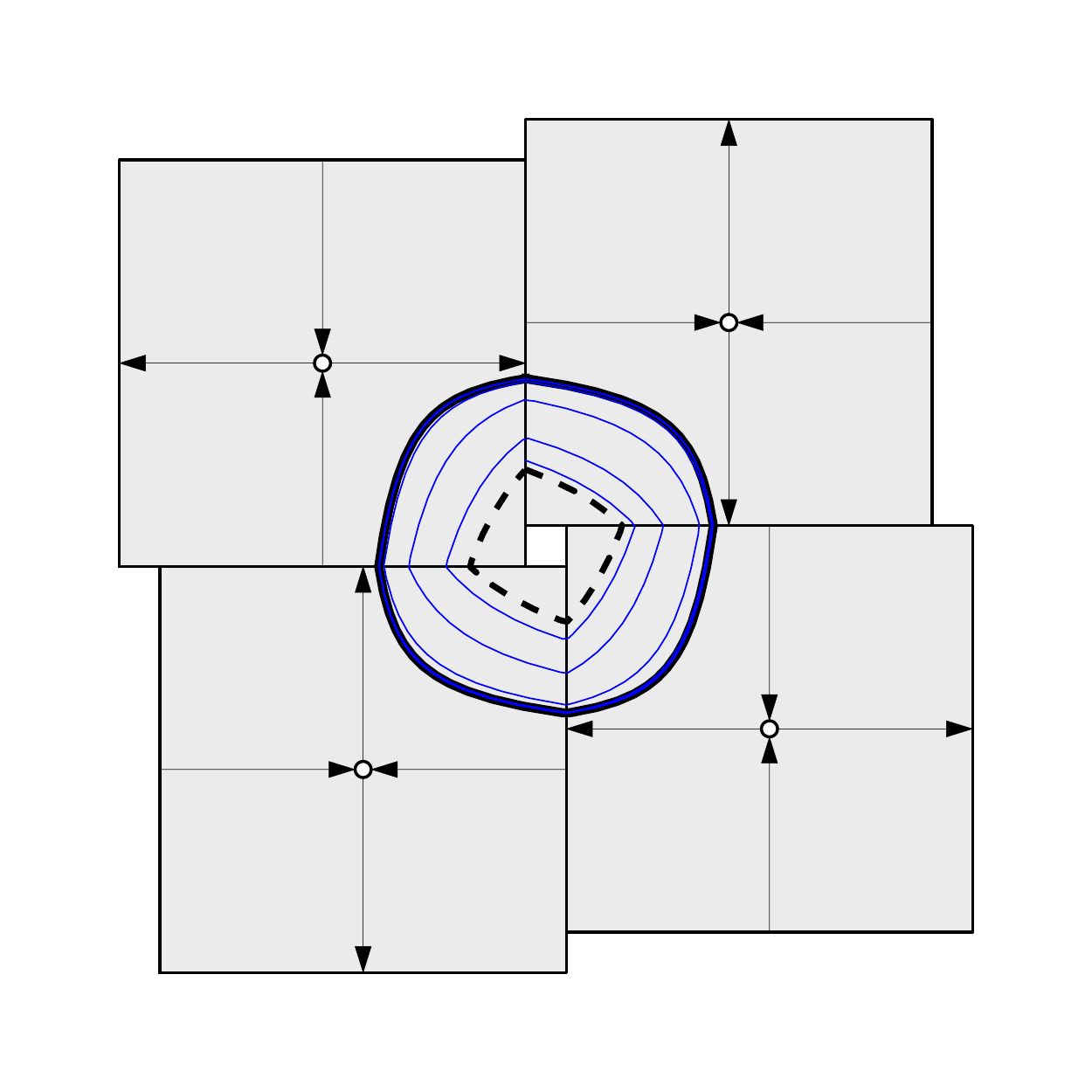}\label{iris_lc1}}
\subfloat[]{\includegraphics[width=.4\figwidth]{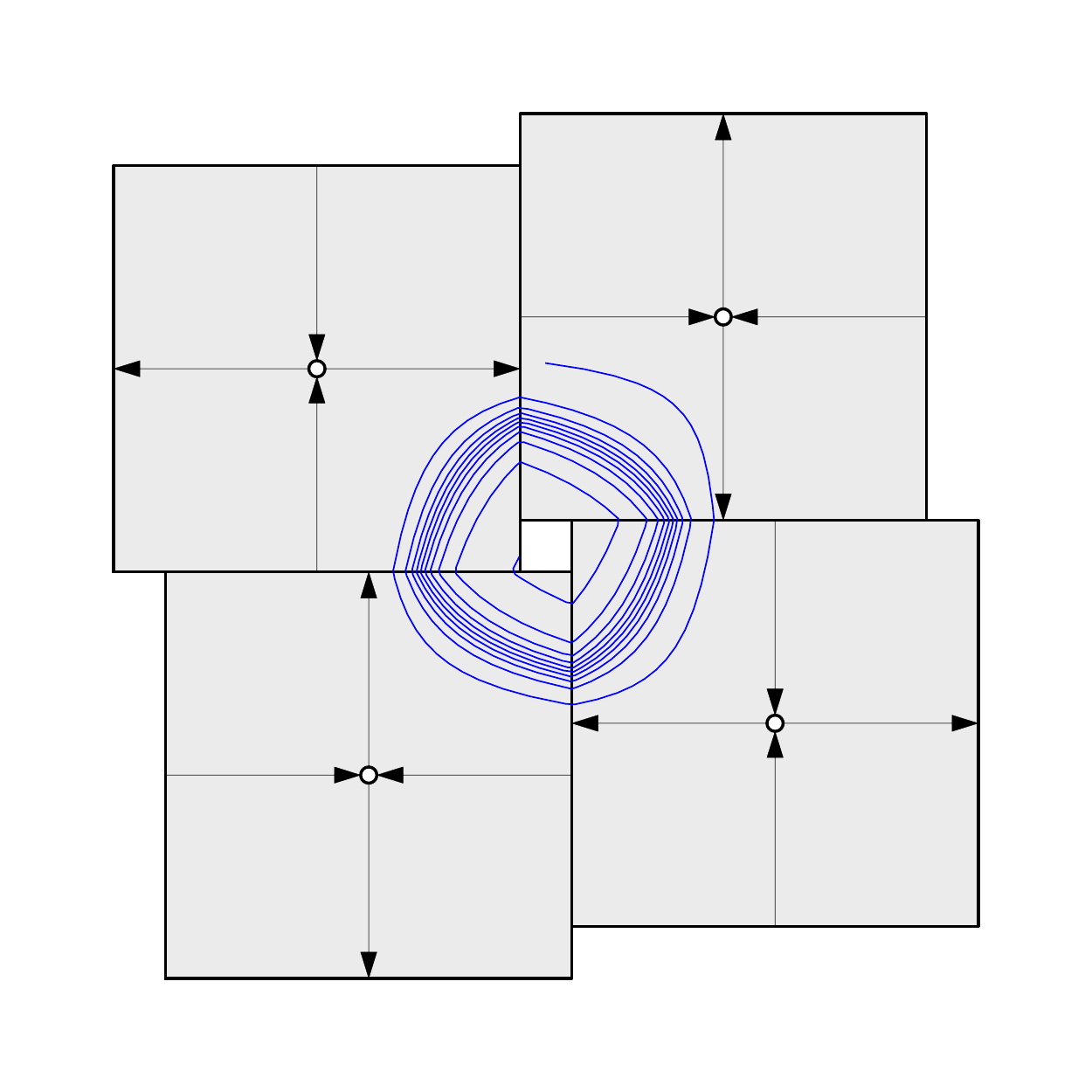}\label{iris_fold}}
\caption{
The behavior of the iris system depends on the offset $a$.  
The system is shown rotated through an angle of $-\tan^{-1}(a/b)$ for easier visual 
comparison with the smooth system.
\subref{iris_shc}~%
[$a=0$]
When the offset $a$ is $0$, the unstable manifold of each saddle connects to
the stable manifold of the next, forming a stable heteroclinic orbit.  (Thin 
blue line: sample trajectory.)
\subref{iris_lc0}~%
[$a=0.05$]
When $0 < a \ll 1$, a stable limit cycle exists that passes close to the 
saddles.  Its basin of attraction extends inward to a small,
unstable limit cycle around the center (dotted black line).  
\subref{iris_lc1}~%
[$a=0.2$]
As $a$ increases, the limit cycle becomes more rounded and the flow along the 
limit cycle more regular.  
\subref{iris_fold}~%
[$a=0.255$]
If $a$ continues to grow, the stable limit cycle is destroyed in a fold 
bifurcation with the inner unstable limit cycle. Please see corresponding movie (file: \texttt{iris.mpg}); this
animation cycles through phase portraits
of the iris system for values of $a$ ranging from 0 to 0.255 and back again.   
}
\label{iris_fig}
\end{figure}

As $\mu\to 0^+$, the period of the limit cycle in system \ref{eq:rotated_smooth_sho} diverges and the asymptotic phase may no longer be defined.  Nevertheless, the response of the system to small, transient perturbations remains of interest.  To explore behavior analogous to phase resetting in the $\mu\to 0^+$ limit, we introduce a piecewise linear planar dynamical system analogous to the smooth system in
Equations~\ref{eq:smooth_sho}.  Figure  \ref{fig:sine_to_iris_fig_sine} illustrates the construction: we tile the plane with large squares of size $2b$
centered on each saddle and smaller squares of size $a$ between them.  When the rotation parameter $\mu$ is zero, the large squares align and the square of side $a$ vanishes.  As the vector field ``rotates'' in the vicinity of each saddle, the squares rotate and $a$ increases from zero.  Within each square, we introduce coordinates $\xi=(s,u)$ with respect to which the local flow obeys $ds/dt=-\lambda s$ and $du/dt=u$.  Here $\lambda>0$ 
is the eigenvalue corresponding to the eigenvector $(1,0)$ tangent to the stable manifold $W^s_{loc}=\{(s,0)| -b<s<b\}$. The unstable manifold  $W^u=\{(0,u)| -b<u<b\}$ is tangent to the eigenvector $(0,1)$ corresponding to the unstable eigenvalue, which without loss of generality is set to one.  
The stable and unstable eigenvectors are arranged so that when $a>0$ the flow on the inner quadrants of each square moves clockwise; the flow leaving the square around the $i^{th}$ saddle enters the square around the $i+1^{st}$ saddle (mod 4), as shown in Figure \ref{fig:sine_to_iris}.

Formally, for $k\in\{1,2,3,4\}$ let the center of the $k^{th}$ square, $(x_k,y_k)$, be given by the real and imaginary parts, respectively, of 
$z_k=\sqrt{2} (-i)^k (be^{-i\pi/4}+(a/2) e^{i\pi/4})$, 
%
where $0\le a<b$.  Define the $k^{th}$ square to be $S_k=\{(x,y):\,|x-x_k|\le b\,\,\,\&\,\,\,|y-y_k|\le b\}$. In the interior of the first square, the flow satisfies $\dot{x}=-\lambda (x-x_1)$ and $\dot{y}=y-y_1$.  The flow in the interior of squares two through four is defined so
that the vector field is equivariant with respect to fourfold rotation about the origin.  Adjacent squares form a boundary of length $(2b-a)$.  At the boundary between $S_k$ and $S_{k+1}$ (mod $k$), the flow leaves $S_k$ and enters $S_{k+1}$, by construction.  The vector field is discontinuous across these boundaries.  For definiteness, we take the flow at points on the mutual boundary $S_k\cap S_{k+1}$ of adjacent squares  to be defined so that the vector field is continuous when approached from square $S_{k+1}$, the square into which the trajectories enter. 

The four offset squares may be repeated to form a partial tiling of the plane, leaving a complementary set composed of smaller squares of size $a$.  We will view the entire system as confined to the 2-torus, however.  We leave the flow unspecified in the interiors of the small squares, and we take any egress points from a large square into a small square to be absorbing, \textit{i.e.}~the flow is set to zero at the boundary of the small squares.  Since we are interested in the phase response of limit cycles whose basins of attraction are entirely contained within the larger squares, the flow in the smaller squares is of no consequence.
Figure \ref{iris_fig} illustrates the system, and shows sample trajectories for different values of $a$.
We will
refer to the piecewise linear system as the \emph{iris system} because it resembles 
the iris of a camera, opening and closing as $a$ increases or decreases.

It is clear that the iris system will form a stable 
heteroclinic orbit when $a=0$ (corresponding to $\mu=0$ in the smooth system);
we will show in \S~\ref{limit_cycle_sec} that it also exhibits stable
limit cycles passing near the saddles for small values of $a > 0$
(corresponding to $\mu>0$) as shown in Figure~\ref{iris_fig}.  
When the limit cycle exists, as above, we assign a phase $\theta$ to each point $\gamma$ in 
the cycle as the fraction of the cycle's period (scaled by $\theta_{\mbox{max}}=4$) required 
to reach that point from a defined starting point $\gamma(0)=\gamma_0$ on the cycle, \textit{i.e.}
\begin{equation}\label{phase_def}
\theta(\nu) = \frac{4}{T} \min\bigl( \{t > 0 \mathrel| \gamma(t) = \nu\} \bigr)
\end{equation}
where $T$ is the period, \textit{i.e.}~$T = \min\bigl(\{t > 0 \mathrel| \gamma(t) = \gamma_0\}\bigr)$.
It is clear that with this transformation $d\theta/dt$ is constant and equal
to $4 / T$.  
We define the point where the limit cycle first enters the bottom left 
square (surrounding saddle $S_1$) as $\gamma_0$.

Because the flow is piecewise linear, we may derive analytically the form of the phase response curve for any values of $\lambda>1>a>0$ for which a stable limit cycle exists.

\begin{thm}\label{thm:main}
Let $\lambda>1>a>0$, let the iris system be defined as in \S \ref{ssec:iris-definition}, and
define the function 
\begin{equation}
\rho(u)=u^{\lambda}-u+a.
\label{eq:roots-function}
\end{equation}
\begin{enumerate}
\item If the function $\rho$
has two isolated positive real roots then the iris system has a stable limit cycle.  Let $u$ denote the smallest positive real root of \ref{eq:roots-function}.  The limit cycle trajectory enters each square at local coordinates $(1,u)$ and exits each square at local coordinates $(s,1)$ where $s=u^{\lambda}$.   The period of the limit cycle is $T=4\log(1/u)$.  
\item Let $\theta \in [0,4)$ be the phase of an instantaneous perturbation in direction $\eta$ and let $\theta=k+\varphi$ where $\varphi\in[0,1)$ 
and $k\in\{0,1,2,3\}$. If $k=0$ then the infinitesimal phase response of the limit cycle is 
\begin{equation}
\label{eq:PRC-thm}
Z(\eta,\varphi,a) = \frac{\eta\cdot\beta(\varphi)}
{\log(1/u)(u-\lambda s)}
\end{equation}
where $\beta(\varphi)=\left(s^{(1-\varphi)},u^\varphi \right)$ and $\eta=(\eta_s,\eta_u)$ is a unit vector in the $L_1$ norm.  If $k\in\{1,2,3\}$ then the infinitesimal phase response is given by $Z(\eta^\prime,\varphi,a)$ where $\eta'=\mathcal{R}^k\eta$ and $\mathcal{R}=\matrix{rr}{0&-1\\1&0}$. 
The magnitudes of the phase responses to perturbations parallel to the stable and unstable eigenvector directions is greatest at phases corresponding to boundaries between piecewise linear regions.
\item As $a\to 0$ for fixed $\lambda$, the entry coordinate scales as $u=a+o(a)$, the infinitesimal phase response to perturbations parallel to the unstable eigendirection in a given square diverges, and the response to perturbations parallel to the stable eigendirection diverges when $\varphi  \in (1-1/\lambda,1)$; otherwise, it converges to zero. 
\end{enumerate}
\end{thm}



Figure \ref{fig:iris_timeplot} illustrates the time course of trajectories of the iris system for $\lambda=2$ and
 $a\in\{10^{-3}, 0.1, 0.2, 0.24\}$. 
 \begin{figure}[htpb] \centering
\includegraphics[width=\figwidth]{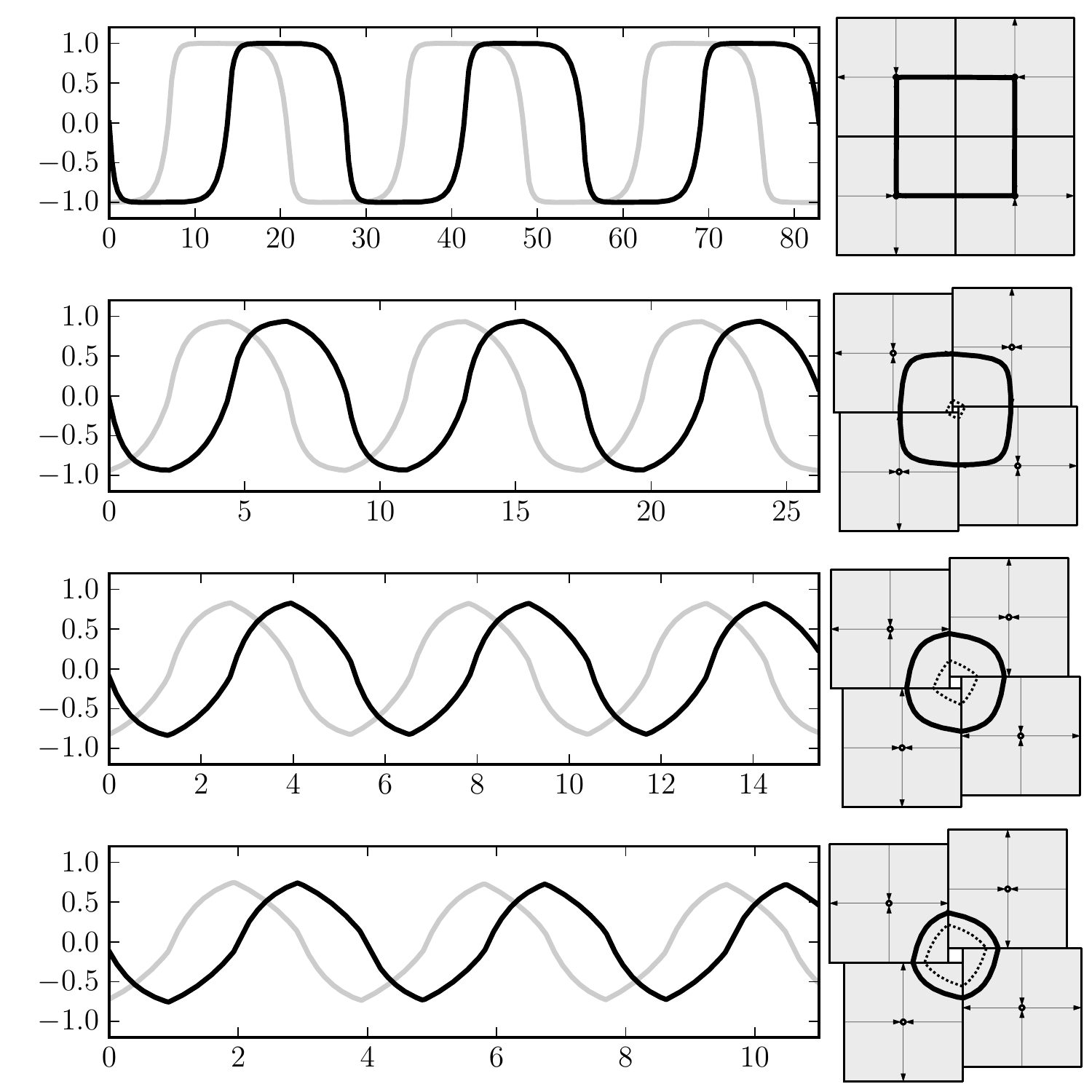}
\caption{
Time plots of limit cycle trajectories of the iris system with 
various values of $a$.  When $a \ll 1$ the trajectories slow dramatically 
when passing near the saddles and travel more quickly between them, resulting 
in time plots with prolonged dwell times.  As $a$ increases, the limit cycle 
moves away from the saddle points and the limit cycle trajectory becomes 
faster and more uniform in speed, resulting in time plots with a more
sinusoidal character.  Note the change in horizontal scale.  The horizontal 
and vertical ordinates are shown in black and gray, respectively, with 
$\lambda = 2$.  From top to bottom $a = 10^{-3}, 0.1, 0.2, 0.24$.  Compare Figure \ref{fig:sine_timeplot}.}
\label{fig:iris_timeplot}
\end{figure}

While a general consideration of phase resetting in the vicinity of a heteroclinic bifurcation on an invariant circle (or, similarly, near a homoclinic bifurcation) is beyond the scope of this paper, it is natural to conjecture that the limiting behavior of phase response curves near such bifurcations may be similar to that observed here. 

Lemmas \ref{lc_existence_lem} and \ref{lem:stability} in \S \ref{sec:iris} provide necessary and sufficient conditions to guarantee the existence of a stable limit cycle.  Lemmas \ref{first_pass_lem}, \ref{lem:subsequent-effects} and \ref{lem:cumulative-change} in \S \ref{sec:local} develop the response of a limit cycle trajectory to a small transient perturbation, leading to direct calculation of the infinitesimal phase response curve. Lemmas \ref{lem:asymptotics1} and \ref{lem:alpha-asympt} in \S \ref{sec:iris-PRC-asympt} describe the asymptotic behavior of the iris system in the heteroclinic limit, \textit{i.e.}~as $a\to 0$.  In \S \ref{sec:iris-PRC-asympt} we also compare analytic and numerical results for the phase response curves of the iris system.  In \S \ref{sec:isochrons} we numerically explore the isochrons of the iris system.  Finally in \S \ref{sec:smooth} we numerically obtain phase response curves for the smooth system given by Equations \ref{eq:smooth_sho} and compare their structure with those of the iris system.


In \BMH Brown \textit{et al.}~studied phase response curves for limit cycles near the four codimension one bifurcations leading to periodic firing in standard neuronal models (saddle-node bifurcation of fixed points on a periodic orbit; supercritical Hopf bifurcation; saddle-node bifurcation of limit cycles; homoclinic bifurcation).  Their analysis of the homoclinic bifurcation corresponds to our analysis of the iris system in a certain limit; for a detailed comparison see \S \ref{ssec:homoclinic}.

\section{Limit Cycles in the Iris System}
\label{sec:iris}

We now prove several results about the iris system that we will need to prove
Theorem~\ref{thm:main}. We start by studying the trajectory within one of the
square regions to construct a map from the time and position of entry into the
region to the time and position of egress out of the region.  Next, we connect
four of the linearized regions together.  We then prove the existence of a
limit cycle for sufficiently small values of $a > 0$.  We also examine the
effects of a perturbation of the trajectory within the neighborhood on the exit
time and exit position.  We use the maps thus derived to show that a
stable limit cycle exists, and use the perturbation results to determine
the asymptotic effect of a small perturbation on the phase of the oscillator
(\textit{i.e.}~the phase response curve).  
First, we will examine the dynamics within a single linear region around a 
saddle of the iris system.   

\begin{rem}[Nondimensionalization]
Assume each saddle of the iris system has two real orthogonal eigenvectors, one stable and one unstable.  
Assume the region around each saddle is a square aligned with 
these vectors and centered on the saddle with a length of $2 L$ along each 
side (so that the saddle is a distance of $L$ away from each edge).  
We will refer to the unstable eigenvalue of the saddle as $\lambda_u>0$ and the
stable eigenvalue as $\lambda_s<0$.  The trajectories within the region have the
dynamics
\begin{align}
    \frac{ds}{dt} & = \lambda_s s, &
    \frac{du}{dt} & = \lambda_u u.
\end{align}
Assume that when a trajectory leaves the edge of one region at position 
$x_f = (s_f, u_f) = (s_f, L)$, it enters the next region with an offset of 
$a$, i.e. $x_i = (s_i, u_i) = (L, s_f + a)$.  We assume that there is a cycle
of four regions the trajectory may traverse this way.  

We can nondimensionalize the system by defining the new state variables 
$s^{*} = s/L$ and $u^{*} = u/L$, by rescaling time as $t^{*} = \lambda_u t$
and the offset as $a* = a/L$.  We may define $\lambda = -\lambda_s/\lambda_u$
as the saddle ratio, which will play a critical role in stability
(\S\ref{limit_cycle_sec}).  Using these definitions, the governing equations
become
\begin{align}
    \frac{ds^{*}}{dt^{*}} & = -\lambda s^{*}, &
    \frac{du^{*}}{dt^{*}} & = u^{*}.
\end{align}
\end{rem}
We will use the non-dimensionalized version of the system for the remainder of
the paper.  For simplicity of notation, we will omit the asterisks in the sequel.

\subsection{Dynamics Within A Linear Region}

Within a single square region surrounding a given saddle point, we will use the local coordinates
$(s,u)$ to represent the displacement parallel to the stable and unstable eigenvectors, respectively.  
The line segments $\{(s,0)\in\R^2|0\le s \le 1\}$ and $\{(0,u)\in\R^2|0\le u \le 1\}$ are part of the stable and unstable manifolds of the saddle, respectively, forming separatrices of the flow in the square.  We will restrict attention to flow entering the square along the edge $\{(1,u)\in\R^2| 0<u<1\}$.  It is easy to see that a trajectory entering at an initial point
$x_i \equiv (s_i, u_i) = (1, u_i)$ at time $t_i$ with $u_i > 0$ will exit at
point $(s_f, u_f) = (s_f, 1) = (u_i^\lambda,1)$ at time $\log(1/u_i)$.  We
define the map $f_l: \mathbb{R} \to \mathbb{R}$ from an entry position
along the $s=1$ edge to an exit position along the $u=1$ edge of the region of
linear flow, and a function $T_1(u_i)$ representing the transit time from an
entry position $(1,u_i)$:  
\begin{align}
    f_l(u_i) &= u_i^\lambda,& T_1(u_i) &= \log(1/u_i).   
\end{align}
Note that $f_l$ is a monotonically increasing function of $u_i$ and that
$T_1$ is a monotonically decreasing function of $u_i$. 

The closest approach of a trajectory to the saddle point occurs when the position $x=(s,u)$ is perpendicular to the velocity $v=(-\lambda s, u)$, or 
$0=-\lambda \exp[-2\lambda t] + u_i^2\exp[2t].$ 
Similarly, one may calculate the time at which the speed of the trajectory is minimal.  These times are
\begin{equation}\label{eq:closest}
t_{\mbox{closest}}=\frac{\log\lambda - 2\log u_i}{2(\lambda+1)},
\hspace{1in}
t_{\mbox{slowest}}=\frac{3\log\lambda - 2\log u_i}{2(\lambda+1)}.
\end{equation}
We can examine the position of closest approach by integrating the system from the entry
position for a time $t_{\mbox{closest}}$, which gives the location
\begin{align*}
    u_{\mathrm{closest}} &= u_i e^{t_{\mathrm{closest}}}
    = u_i \left({\lambda}/{u_i^2}\right)^{1/(2\lambda + 2)} \\
    s_{\mathrm{closest}} &= e^{-\lambda t_{\mathrm{closest}}}
    = \left({u_i^2}/{\lambda}\right)^{\lambda/(2\lambda + 2)} 
\end{align*}
Taking the ratio of the two coordinates we find
\begin{align}
    u_{\mathrm{closest}} = \lambda^{1/2}\ s_{\mathrm{closest}},
\end{align}
and thus the point of closest approach of each trajectory in this quadrant lies along 
a line passing through the saddle.  A similar calculation using $t_{\mathrm{slowest}}$
shows that the slowest points lie along the line
\begin{align}
    u_{\mathrm{slowest}} = \lambda^{3/2}\ s_{\mathrm{slowest}}
\end{align}
which also passes through the saddle point at the origin of the local coordinate system.

\subsection{Dynamics Across Regions} \label{limit_cycle_sec} 

We will next explore the conditions under which the iris system contains a 
limit cycle. 

\begin{lem} 
\label{lc_existence_lem}
The iris system described in \S\ref{ssec:iris-definition} contains a   limit cycle iff 
the function
\[
\rho(u) = u^\lambda - u + a 
\] 
has isolated real roots  $u^\dagger, u^\ddagger \in (0,1)$, with 
$u^\dagger \le u^\ddagger$.  These roots exist iff $\lambda > 1$ and 
\[
\lambda^{\lambda / (1-\lambda)} - \lambda^{1/(1-\lambda)} + a \le 0.
\] 

\end{lem}

\begin{proof}
To track crossings between regions, we will add a subscript to the variable
names indicating how many region crossings have occurred since time $t=0$, so
that $u_{i,3}$ is the ingress location ($u_i$) immediately after the third crossing.  

When $a > 0$, trajectories leaving one region enter the
next with an offset $a$, \textit{i.e.}~$(s_{i, n+1}, u_{i,n+1}) = (1, s_{f,n} + a)$.  
We therefore have a second monotonically increasing map from the point of
egress along the edge of one square to the entry point along the edge of the 
next square, $f_e: s_{f,n} \to u_{i,n+1}=s_{f,n}+a$.  

We can now form a map from the entry position along the edge of one linear
region to the entry position along the edge of the next region \textit{via} the composition
\begin{equation}
h(u_i) = (f_e \circ f_l)(u_i) = (u_i)^\lambda + a.
\end{equation}
As the composition of two monotonically increasing maps, $h$ is also 
monotonically increasing.  Noting that the entry edges form a transversal 
section of the flow, it is clear that $h^4$, the fourfold composition of $h$,
forms a Poincar\'e map for any cycles crossing this edge.  Limit cycles will 
form isolated fixed points in this map, so to find potential limit cycles we
look for points where $u_i = h^4(u_i)$.  Because $h$ is composed of
four identical monotonic maps, the limit cycles will also be fixed points in
these component maps and thus we have a fixed point at $u_i=u$ if and only if $u = h(u) = u^\lambda + a$, 
or equivalently $
\rho(u) \equiv u^\lambda - u + a = 0,  
$   (\textit{cf.}~Equation \ref{eq:roots-function}).

We now consider the potential values of $\lambda$, which by definition must be positive.  First, consider 
$0<\lambda < 1$.  Rewriting $\rho(u)$ as $u^\lambda(1 - u^{1-\lambda}) + a$
and remembering that $u \in (0,1)$, we can see that $\rho(u)$ is now the sum of a
positive product and a non-negative parameter, and thus cannot be zero.  Next,
consider $\lambda = 1$.  In this case $\rho(u)$ becomes just $a$, and thus 
$\rho = 0$ implies $a = 0$.  Note that $\rho$ is independent of $u$ and thus every
$u \in (0,1)$ is a solution, corresponding to a one parameter family 
of neutrally stable, non-isolated orbits. 
Therefore $\lambda$ must be greater than one if a limit cycle exists.  

Differentiating $\rho$ twice with respect to $u$, we find that $\rho$ is convex for all $u > 0$ when 
$\lambda > 1$, and differentiating it once we find that $\rho$ has
a minimum at 
\begin{equation}
\label{u_i_min_eqn}
(u_i)_{\min} = \lambda^{1/(1-\lambda)}.
\end{equation}  
Because 
$\rho(u)$ is continuous, and equal to $a$ when $u = 0$ or $u = 1$, by the midpoint theorem it
will have a root between $0$ and $(u_i)_{\min}$ and a root between
$(u_i)_{\min}$ and $1$ iff
\begin{equation}
\label{roottest}
0 \ge r((u_i)_{\min}) = (u_i)_{\min}^\lambda - (u_i)_{\min} + a 
    = \lambda^{\lambda / (1-\lambda)} - \lambda^{1/(1-\lambda)} + a.
\end{equation}
We will call the smaller root $\ud$ and the larger (possibly identical)
root $\udd$.
\end{proof}

\begin{figure} \centering
\includegraphics[width=.7\figwidth]{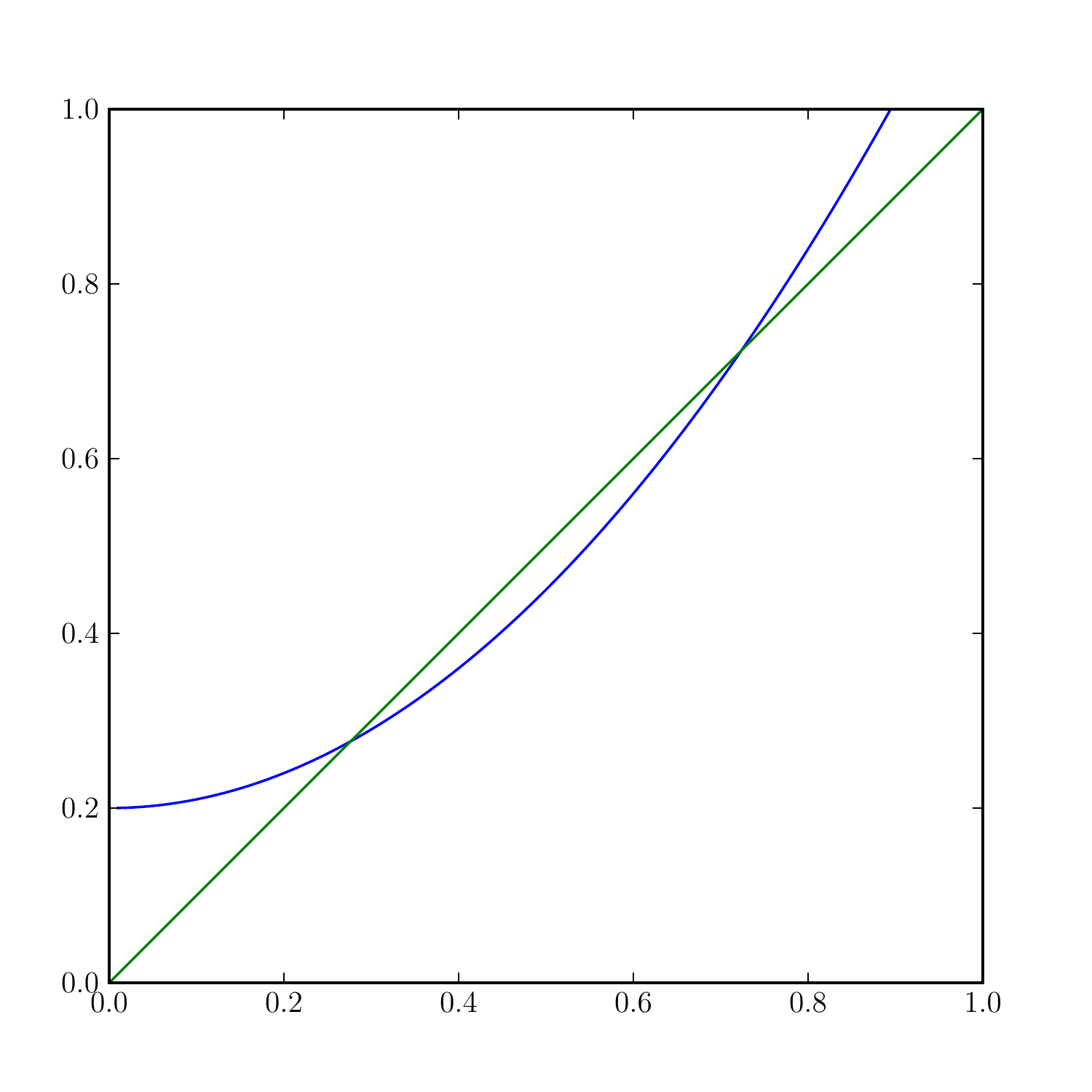}
\caption{The map $h=f_e\circ f_l$ from the entry position of a trajectory along the edge a 
square to its entry position along the edge the next square forms a map
analogous to a return map (blue curve).  This map is continuous and
monotonically increasing, so a limit cycle can only exist where a trajectory
enters the next square at the same position it entered the current square
(green line indicates the identity map).  In this example there are two intersections corresponding to
the stable and unstable limit cycles.  Here $a=0.2$ and $\lambda=2$.  
The blue curve meets the vertical 
axis at the offset between squares $a$, and changing the offset raises or 
lowers the blue curve without changing its shape.  Varying $a$ allows for $0$, $2$, or
$1$ limit cycles (corresponding to Figure~\ref{iris_fold}, 
Figure~\ref{iris_lc1}, and the fold bifurcation that occurs between them).}
\label{iris_return_map_fig}
\end{figure}

Figure \ref{iris_return_map_fig} illustrates the map $h=f_e\circ f_l$ when $\lambda=2$ and $a=0.2$.

We now examine the stability of the roots $\ud$ and $\udd$.  

\begin{lem}\label{lem:stability}
If the roots $\ud$ and $\udd$ of Equation \ref{eq:roots-function} exist and are distinct, $\ud$ 
gives the entry position of a stable limit cycle along the $s=1$ edge of a 
region and $\udd$ gives the entry position of an unstable limit cycle 
along the same edge.  
\end{lem}

\begin{proof}
We can determine the stability of the two points by examining 
the derivative of $h$ at these points; if $|dh/du_i| > 1$ the fixed point will
be unstable, and if $|dh/du_i| < 1$ the point will be stable.  Because the
existence of the roots implies $\lambda > 1$, we know $dh/du_i = (u_i)^{\lambda
- 1} \lambda$ will be positive.  Differentiating $h$ a second time we see that
$d^2h/du_i^2 = (u_i)^{\lambda - 1} (\lambda - 1) \lambda$ is always greater
than zero, and thus $dh/du_i$ is monotonically increasing.  Since a fixed point
in a map is stable iff the magnitude of the derivative at that point is greater
than one, we find the point where $dh/du_i$ passes through $1$:
\[
1 = dh/du_i = \lambda (u_i)^{\lambda - 1}
\]
or
\[
u_i = \lambda^{1/(1 - \lambda)}.
\]
This, however, is just $(u_i)_{\min}$ from Equation~\ref{u_i_min_eqn}.  Because
this lies between the two roots, $\left.dh/du_i\right|_{\ud} <
\left.dh/du_i\right|_{(u_i)_{\min}} = 1$, and thus $\ud$ is a stable
fixed point of $h$, and $1 = \left.dh/du_i\right|_{(u_i)_{\min}} <
\left.dh/du_i\right|_{\udd}$ and thus $\udd$ is an unstable
fixed point of $h$.  Because $h$ is a Poincar\'e map at the entry plane of the
region, these fixed points correspond to the entry points of a stable and
unstable limit cycle respectively.  \end{proof}

\begin{rem}
Lemmas \ref{lc_existence_lem} and \ref{lem:stability} together establish Theorem \ref{thm:main}, Part 1.
\end{rem}

As just shown, for small positive $a$ two limit cycles exist - one stable and one unstable.  As 
$a \to 0$, the stable limit cycle is destroyed in a heteroclinic bifucation 
and the unstable limit cycle collapses to a point at the center of the system.
As $a$ increases, the two limit cycles collide in a cycle-cycle fold bifurcation.  The collision occurs when the inequality \eqref{roottest} becomes an equality.  The two roots converge to 
$(u_i)_{\min}$ 
as
the two limit cycles that cross the section at the two roots
merge.  
Figure \ref{iris_bifurcation_fig} shows the bifurcation diagram.

\begin{figure} \centering
\includegraphics[width=0.7\figwidth]{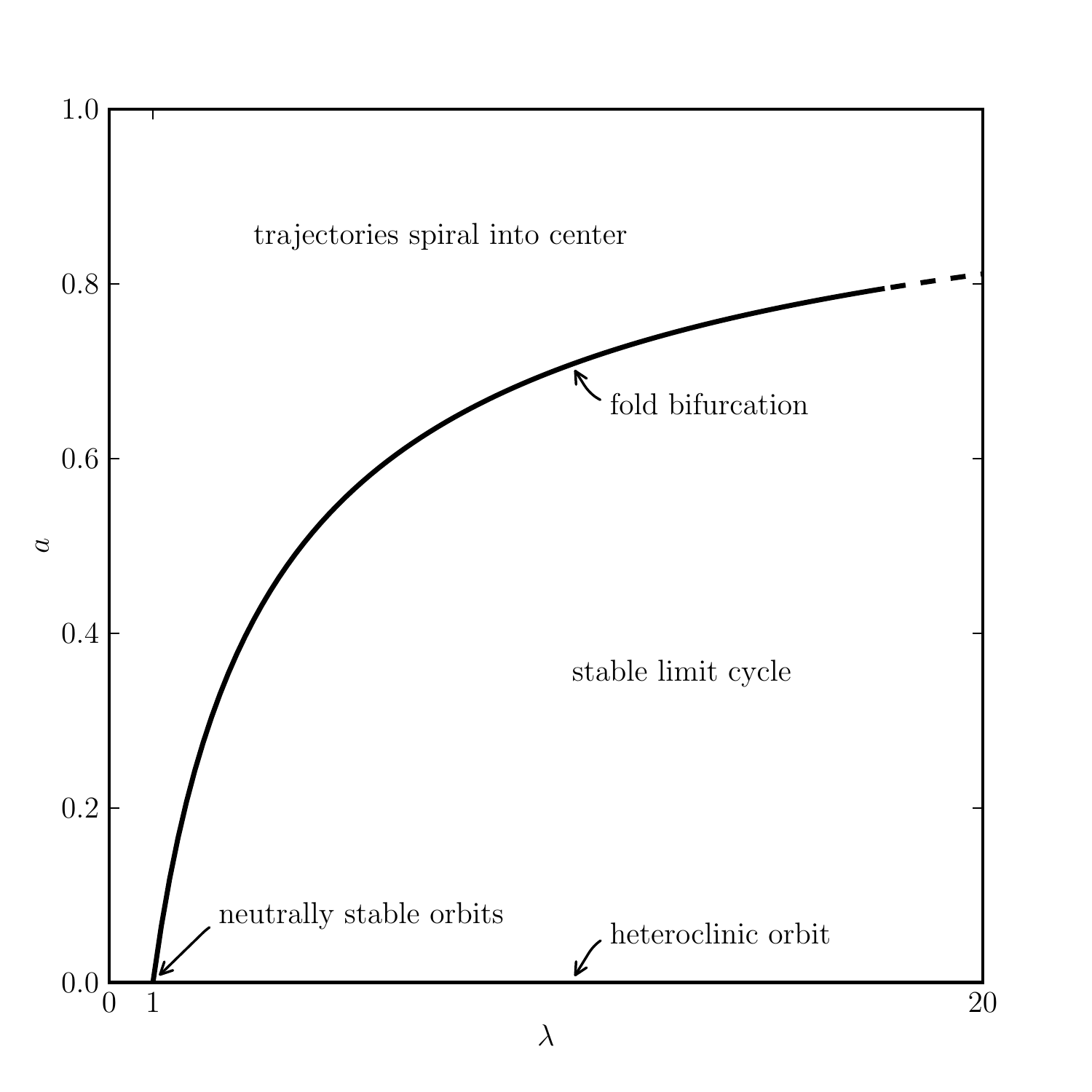}
\caption{Two-parameter bifurcation diagram for the iris system.  When the stable-to-unstable eigenvalue ratio $\lambda$ is
sufficiently large relative to $a$, a stable and an unstable limit cycle coexist 
(region labeled `stable limit cycle').  As $a$ increases the two limit
cycles are destroyed in a fold bifurcation (line labeled `fold bifurcation').
As $a$ approaches $0$, the system approaches a heteroclinic orbit (line labeled
`heteroclinic orbit').  When $a$ is sufficiently large or $\lambda < 1$,
trajectories spiral into the center as shown in Figure~\ref{iris_fold} (area
labeled `trajectories spirals into center').  When $\lambda=1$ and $a=0$, the system becomes
a square filled with neutrally stable orbits (point labeled `neutrally stable
orbits').} 
\label{iris_bifurcation_fig} 
\end{figure}

\section{Effects of a small instantaneous perturbation}
\label{sec:local}

We now examine the effects of a small perturbation of a trajectory starting on
the limit cycle.  We will first explore the effects on the exit time and exit
position from the square, and then extend our analysis to the asymptotic 
effect of the perturbation on phase over many cycles (the phase response curve or PRC).

\subsection{Initial effects of a small perturbation}

Assume a trajectory begins on the stable limit cycle
at position $x_0 = (1, \ud)$ at time $t=0$, and receives a small
perturbation of size $r$ in the direction of a unit vector 
$\eta = (\eta_u, \eta_s)$ at time $t^\prime$.  We will 
consider only perturbations that occur before the trajectory leaves the square 
(\textit{i.e.}~$0 \le t^\prime < T_1(\ud) = T_1^\dagger$);  
earlier or later perturbations can be reduced to this case by
an appropriate time shift and discrete rotation.  We will also require that the perturbation not push the 
trajectory out of the basin of attraction of the limit cycle.  The infinitesimal PRC we calculate below
is obtained in the limit of small perturbations; for finite $a>0$, the basin of attraction has finite width, and trajectories that escape the limit cycle's attracting set are not assigned an asymptotic phase.

It is straightforward to 
calculate the effects of the perturbation on the time and 
position at which the trajectory leaves the square.  We first consider the
typical case where the perturbation does not push the trajectory across 
the border of the square.

\begin{lem}\label{first_pass_lem}
Consider a trajectory initially on the stable limit cycle $\gamma$ of the iris system such that $\gamma(0)=(1,\ud)$,
and a perturbation $\Delta x$ of size $||\Delta x||=r$ in the direction of unit vector $\eta=(\eta_s,\eta_u)=\Delta x/r$ at time $0<t^\prime < T_1^\dagger$. 
Assume
the perturbation does not push the trajectory out of the basin of 
attraction or into another square.  Then the perturbation will result in a
change in the position of entry to the next square
\begin{equation}\label{first_pass_du}
\Delta u_1^\prime = \big(\eta_s (\ud e^{t^\prime})^{\lambda} 
    + \eta_u \lambda (\ud)^{\lambda-1} e^{-t^\prime} \big) r + o(r)
\end{equation}
and a change in transit time of the square
\begin{equation}\label{first_pass_dt}
\Delta T_1^\prime = -(\eta_u e^{-t^\prime} / \ud) r + o(r),
\end{equation}
as $r\to 0$.
\end{lem}

\begin{proof}
Immediately before the 
perturbation occurs, the trajectory will be at the point
\[
    \lim_{t \to t^\prime-} x(t) = (e^{-\lambda t^\prime}, \ud 
        e^{t^\prime}).
\]
The perturbation shifts the position to
\[
    \lim_{t \to t^\prime+} x(t) = (e^{-\lambda t^\prime} + r \eta_s, 
        \ud e^{t^\prime} + r \eta_u).
\]
The trajectory will exit the square once $x_u$ grows to $1$ at time
\[
T_1^\prime = t^\prime + \bigl(- \log(\ud e^{t^\prime} + r \eta_u)\bigr)
\] 
and enter the next square along the $s = 1$ edge at position
\[
u_1^\prime =  
        a + (e^{-\lambda t^\prime} + r \eta_s) 
                e^{-\lambda(T_1^\prime - t^\prime)} 
    = a + (e^{-\lambda t^\prime} + r \eta_s) 
            (\ud e^{t^\prime} + r \eta_u)^{\lambda}
\]
For $r \ll 1$, we may write 
\begin{align*}
T_1^\prime &= t^\prime - \log(\ud e^{t^\prime})
    -  (\eta_u e^{-t^\prime} / \ud) r + o(r),\\
u_1^\prime &= a + (\ud)^\lambda 
    +  \big(\eta_s (\ud e^{t^\prime})^{\lambda} 
    + \eta_u \lambda (\ud)^{\lambda-1} e^{-t^\prime} \big) r + o(r)
\end{align*}
as $r \to 0$.  Noting that 
$t^\prime - \log(\ud e^{t^\prime}) = -\log(u_i) = T_1^\dagger$ and
$a + (\ud)^\lambda = \ud$, these expressions may be viewed as the exit time and 
position for the limit cycle with a linear correction and 
additional higher order terms in $r$. That is, 
\begin{align*}
T_1^\prime &= T_1^\dagger - (\eta_u e^{-t^\prime} / \ud) r + o(r),\\
u_1^\prime &= \ud + \big(\eta_s (\ud e^{t^\prime})^{\lambda} 
    + \eta_u \lambda (\ud)^{\lambda-1} e^{-t^\prime} \big) r + o(r).
\end{align*}
Subtacting $T_1^\dagger$ and $\ud$ respectively then gives our 
result.
\end{proof}

Perturbations that push the trajectory across the boundary between two squares
can be handled without much additional difficulty.  For example, we can 
treat a perturbation that crosses into the next square as an advancement of 
the trajectory to the point that it enters the next square followed by a 
perturbation to the new location in the new square.  
Perturbations to the previous square can be handled in a similar fashion, as can perturbations to
the diagonal square.  
However, for any finite $a>0$ and any perturbation time other than $t^\prime=0$ or $t^\prime=T_1^\dagger$, the perturbed point  will remain in the same square as the unperturbed point for sufficiently small $r>0$. Therefore the additional cases do not alter the calculation of the PRC except possibly at a set of finite points, and they will be omitted.

\subsection{Subsequent effects of a perturbation}

In the absence of perturbation, the limit cycle passing through $\gamma(0)=(1,\ud)$
will exhibit a (constant) sequence of edge crossing locations $u_{i,n}=\ud, n=\{0,1,2,\cdots\}$
at a sequence of crossing times $T^\dagger_n=nT^\dagger_1$ at constant intervals $\Delta T^\dagger_n\equiv T^\dagger_1$.  A single perturbation $\Delta x=r(\eta_s,\eta_u)$ at time $t^\prime$ will lead to a new sequence of edge crossings $\{u_{i,n}^\prime\}$ and crossing times $\{T_n^\prime\}$. 
Following the perturbation, the time spent between the $n^{th}$ and $n+1^{st}$ crossings is offset from the unperturbed dwell time: $T_{n+1}^\prime-T_n^\prime=T_1^\dagger+\Delta T_{n+1}^\prime$.  
Next we calculate the change in entry position for subsequent edge 
crossings ($\Delta u_n^\prime=u_{i,n}^\prime-\ud$) and the dwell time offsets for subsequent 
squares ($\Delta T_{n+1}^\prime$).  

\begin{lem}\label{lem:subsequent-effects}
Under the conditions of Lemma~\ref{first_pass_lem}, the entry position after
the $n^{th}$ crossing $(n \ge 1)$ will be offset by 
\begin{equation}\label{nth_pass_du}
\Delta u_n^\prime = \left(\lambda(\ud)^{\lambda-1}\right)^{n-1} 
        \big(\eta_s (\ud e^{t^\prime})^{\lambda} 
            + \eta_u \lambda (\ud)^{\lambda-1} e^{-t^\prime} \big) r 
    + o(r),
\end{equation}
and the time spent in that square will be offset by 
\begin{equation}\label{nth_pass_dt}
\Delta T_{n+1}^\prime = (\ud)^{-1} 
        \left(\lambda(\ud)^{\lambda-1}\right)^{n-1} 
        \big(\eta_s (\ud e^{t^\prime})^{\lambda} 
            + \eta_u \lambda (\ud)^{\lambda-1} e^{-t^\prime} \big) r 
     + o(r),
\end{equation}
as $r\to 0$.
\end{lem}

\begin{proof}
A perturbed trajectory will enter the $n^{th}$ edge with an offset 
$\Delta u_n^\prime$.  This situation is equivalent to a 
trajectory entering along the limit cycle that experiences 
a perturbation in the $u$ direction ($\eta = (0,1)$) of magnitude 
$r = \Delta u_n^\prime$ immediately upon entering the square.  We can thus 
use Equation~\ref{first_pass_du} with $t^\prime = 0$ to find the entry 
position along the next ($n+1^{st}$) edge:
\[
\Delta u_{n+1}^\prime = \lambda(\ud)^{\lambda-1} \Delta u_n^\prime 
    + o(\Delta u_n^\prime).
\]
Therefore to leading order the sequence of offsets follows a geometric series, 
with closed form
\begin{align*}
\Delta u_n^\prime &= \left(\lambda(\ud)^{\lambda-1}\right)^{n-1} 
        \Delta u_1^\prime + o(\Delta u_1^\prime) \\
    &= \left(\lambda(\ud)^{\lambda-1}\right)^{n-1} 
        \big(\eta_s (\ud e^{t^\prime})^{\lambda} 
            + \eta_u \lambda (\ud)^{\lambda-1} e^{-t^\prime} \big) r 
    + o(r),
\end{align*}
as $r\to 0$.
Substitution into Equation~\ref{first_pass_dt} provides
the deviation in the time of passage through each square, compared with that for
the limit cycle trajectory.
Again setting $\eta=(0,1)$, $r = \Delta u_n^\prime$, and $t'=0$ gives
\begin{align*}
\Delta T_{n+1}^\prime &= 
        -(\ud)^{-1} \Delta u_n^\prime + o(u_n^\prime) \\
    &= -(\ud)^{-1} 
        \left(\lambda(\ud)^{\lambda-1}\right)^{n-1} 
        \big(\eta_s (\ud e^{t^\prime})^{\lambda} 
            + \eta_u \lambda (\ud)^{\lambda-1} e^{-t^\prime} \big) r 
     + o(r)
.
\end{align*}
\end{proof}

We now examine the asymptotic effect of the perturbation at long times.  
Because $\ud$ is less than one, both 
$\Delta u$ and $\Delta t$ will approach zero for large $n$.  This is 
equivalent to saying that the orbit will asymptotically return the stable 
limit cycle, as expected.  It may, however, return with a different phase
than the unperturbed trajectory.  

\begin{lem}\label{lem:cumulative-change}
Under the conditions of Lemma~\ref{first_pass_lem}, the cumulative 
change in crossing times after $n$ crossings is
\begin{equation}
\Delta C_n^\prime = 
        -\frac{
            \eta_s (\ud e^{t^\prime})^{\lambda} 
                \left(1-
                \left(\lambda (\ud)^{\lambda-1}\right)^{n+1}\right)
            + \eta_u  e^{-t^\prime} \left(1-
                \left(\lambda (\ud)^{\lambda-1}\right)^{n+2}\right)
        }{\ud-\lambda (\ud)^{\lambda}}
     r
    + o(r)
\end{equation}
which converges to 
\begin{equation}\label{c_infty_def}
\Delta C_\infty^\prime = 
        -\frac{
            \eta_s (\ud e^{t^\prime})^{\lambda} 
            + \eta_u  e^{-t^\prime} 
        }{\ud-\lambda (\ud)^{\lambda}}
     r
    + o(r)
    .
\end{equation}
as $n \to \infty$.
\end{lem}
\begin{proof}
The cumulative change in crossing times is the sum of the dwell
time offsets in the previous squares:
\[
\Delta C_n^\prime = \sum_{k=1}^n \Delta T_k^\prime 
    = \Delta T_1^\prime + \sum_{k=2}^n \Delta T_k^\prime 
\]
Substituting in Equations~\ref{first_pass_dt} and~\ref{nth_pass_dt} yields the result: 
\begin{align*}
\Delta C_n^\prime &=
    -\frac{\eta_u e^{-t^\prime}}{\ud} r 
    + \sum_{k=2}^n
        \frac{-1}{\ud} 
            \left(\lambda(\ud)^{\lambda-1}\right)^{k-2} 
            \big(\eta_s (\ud e^{t^\prime})^{\lambda} 
                + \eta_u \lambda (\ud)^{\lambda-1} e^{-t^\prime} \big)
    + o(r)
  \\ &= 
    \frac{-1}{\ud} \left( 
    \eta_u e^{-t^\prime} 
    + \big(\eta_s (\ud e^{t^\prime})^{\lambda} 
            + \eta_u \lambda (\ud)^{\lambda-1} e^{-t^\prime} \big)
        \sum_{k=0}^n
        \left(\lambda (\ud)^{\lambda-1}\right)^k 
    \right) r
    + o(r)
  \\ &= 
    \frac{-1}{\ud} \left( 
    \eta_u e^{-t^\prime} 
    + \big(\eta_s (\ud e^{t^\prime})^{\lambda} 
            + \eta_u \lambda (\ud)^{\lambda-1} e^{-t^\prime} \big)
        \frac{1-\left(\lambda (\ud)^{\lambda-1}\right)^{n+1}}%
            {1-\lambda (\ud)^{\lambda-1}}
    \right) r
    + o(r)
  \\ &= 
        -\frac{
            \eta_s (\ud e^{t^\prime})^{\lambda} 
                \left(1-
                \left(\lambda (\ud)^{\lambda-1}\right)^{n+1}\right)
            + \eta_u  e^{-t^\prime} \left(1-
                \left(\lambda (\ud)^{\lambda-1}\right)^{n+2}\right)
        }{\ud-\lambda (\ud)^{\lambda}}
     r
    + o(r)
    .
\end{align*}
Recalling that $0<\ud < 1$ and $\lambda > 1$, taking the limit as 
$n\to\infty$ gives
\[
\Delta C_\infty^\prime = 
        -\frac{
            \eta_s (\ud e^{t^\prime})^{\lambda} 
            + \eta_u  e^{-t^\prime} 
        }{\ud-\lambda (\ud)^{\lambda}}
     r
    + o(r)
\] as required. \end{proof}

The limit cycle enters each square at location $(1,\ud)$ and exits at $((\ud)^\lambda,1)$, so it is natural to define the exit coordinate along the stable eigenvector axis as $\sd=(\ud)^\lambda$.  Recalling that $\ud=\sd+a$, Equation \ref{c_infty_def} may be simplified as
\begin{equation}
\label{eq:C-simpler}
\Delta C_\infty^\prime = 
        -\frac{
            \eta_s (\sd e^{\lambda t^\prime}) 
            + \eta_u  e^{-t^\prime} 
        }{\ud-\lambda \sd}
     r
    + o(r)
    =
      -\frac{\eta\cdot\beta
        }{\ud-\lambda\sd}
     r
    + o(r).
\end{equation}
The vector $\beta=\left(\sd e^{\lambda t^\prime} , e^{-t^\prime}  \right)$ may be thought of as a solution to  the same differential equation as a trajectory within the first square, $\xi=(s,u)$, but with time reversed and initial conditions set to the far end of the square.  That is,
\begin{eqnarray*}
\frac{d\xi}{dt}&=&\matrix{cc}{-\lambda & 0 \\ 0 & 1},\,\,\,
\xi(0)=\matrix{c}{1\\ \ud}\\
\frac{d\beta}{dt}&=&\matrix{cc}{\lambda & 0 \\ 0 & -1},\,\,\,
\beta(0)=\matrix{c}{\sd \\ 1}.
\end{eqnarray*}

\subsection{Infinitesimal phase response curve}
\label{sec:iPRC}
We now derive the infinitesimal phase response curve.
Initially, we restrict attention to perturbations occurring within the first
quarter cycle, \textit{i.e.}~within the first square.  The PRC for perturbations occurring after $k$
additional border crossings is obtained \textit{via} a discrete rotation operation.  

We define the phase of a 
point on the limit cycle as in Equation~\ref{phase_def}, \textit{i.e.}~as the smallest amount of time required to reach that
point from the point where the limit cycle enters the first square.  
Therefore the asymptotic shift in crossing times translates into an asymptotic shift in phase as
\[
\Delta\varphi = \frac{4}{T}\Delta C_\infty^\prime = \frac{\Delta C_\infty^\prime}{\log(1/\ud)},
\]
recalling that $T=4T_1^\dagger=4\log(1/\ud)$.  
Since 
$e^{-t}=(\ud)^\varphi$ 
we may write $\beta$ in the form 
\begin{equation}
\label{eq:beta-def}
\beta(\varphi)=\left(\beta_s(\varphi),\beta_u(\varphi)\right)
=\left((\ud)^{\lambda(1-\varphi)},(\ud)^\varphi \right)
=\left((\sd)^{(1-\varphi)},(\ud)^\varphi \right)
,
\end{equation}
simplifying the subsequent analysis.
Thus the asymptotic change in phase for a perturbation at phase $\varphi\in[0,1)$
is  
$$
\Delta \varphi = \frac{\eta\cdot\beta(\varphi)}
{\log(1/\ud)(\ud-\lambda\sd)}r+o(r).
$$
The infinitesimal phase response curve, obtained by taking the limit of $\Delta\varphi/r$ as $r\to 0$, is therefore
$$Z(\eta,\varphi,a) = \frac{\eta\cdot\beta(\varphi)}
{\log(1/\ud)(\ud-\lambda\sd)},$$
which is Equation \ref{eq:PRC-thm}.

For a perturbation at phase $\varphi\ge 1$, let $k$ be the number of boundary crossings
preceding the perturbation, \textit{i.e.}~the positive integer satisfying $\varphi\in[k,k+1)$. 
Define a $k$-fold quarter rotation operation on the phase variable and the direction of the perturbation by 
\begin{align}
\varphi &\to\varphi^\prime=\varphi-k\\
\eta &\to \eta^\prime=\mathcal{R}^k\eta
\end{align}
where $\mathcal{R}=\left(\begin{array}{rr}0&-1\\1&0\end{array} \right)$.
Then the infinitesimal phase response due to a perturbation in direction $\eta$ at phase $\theta>0$ is given by 
$Z(\eta^\prime,\varphi^\prime,a)$.

The dependence of the PRC on the phase is entirely contained in the term $\eta\cdot\beta$.  Consider the components of the response to perturbation in the horizontal and vertical directions.  When $\varphi\in[0,1)$ the component $\beta_s$ along the direction of the stable eigenvector corresponds to $\beta_x$, the component in the horizontal direction.  Similarly the components along the unstable direction $\beta_u$ and the 
vertical
component $\beta_y$ are identical.  Moving clockwise around one full orbit we observe that:
$$\begin{array}{lllll}
\mbox{when} & \varphi\in[0,1), & \beta_x = \beta_s & \mbox{and} & \beta_y=\beta_u;\\
\mbox{when} & \varphi\in[1,2), & \beta_x = \beta_u & \mbox{and} & \beta_y=-\beta_s;\\
\mbox{when} & \varphi\in[2,3), & \beta_x = -\beta_s & \mbox{and} & \beta_y=-\beta_u;\\
\mbox{and when} & \varphi\in[3,4), & \beta_x = -\beta_u & \mbox{and} & \beta_y=\beta_s.\\
\end{array}$$
For each component the resulting phase response curve consists of two continuous segments separated by a jump of size $(\ud+\sd)$, with antisymmetry $\beta(\varphi+2)=-\beta(\varphi)$. 
The horizontal response component, $\beta_x$, peaks at $\beta_x(1)=1$ and is strictly positive for $\varphi\in[0,2)$.  Similarly, $\beta_y$ peaks at $\beta_y(0)=1$ and is strictly positive for $\varphi\in[0,1)\cup [3,4)$. 
Hence the extrema of both components occur at phase values corresponding to the boundaries between adjacent squares.
\begin{rem}
This concludes the proof of Theorem \ref{thm:main}, Part 2.
\end{rem}

Figure \ref{iris_prc_fig} shows examples of the full iPRCs as a function of $\theta\in[0,4)$ for $\lambda=2$ and $a\in\{10^{-3}, 0.1, 0.2, 0.24\}$.   
\begin{figure}[htbp] \centering
\includegraphics[width=.9\figwidth]{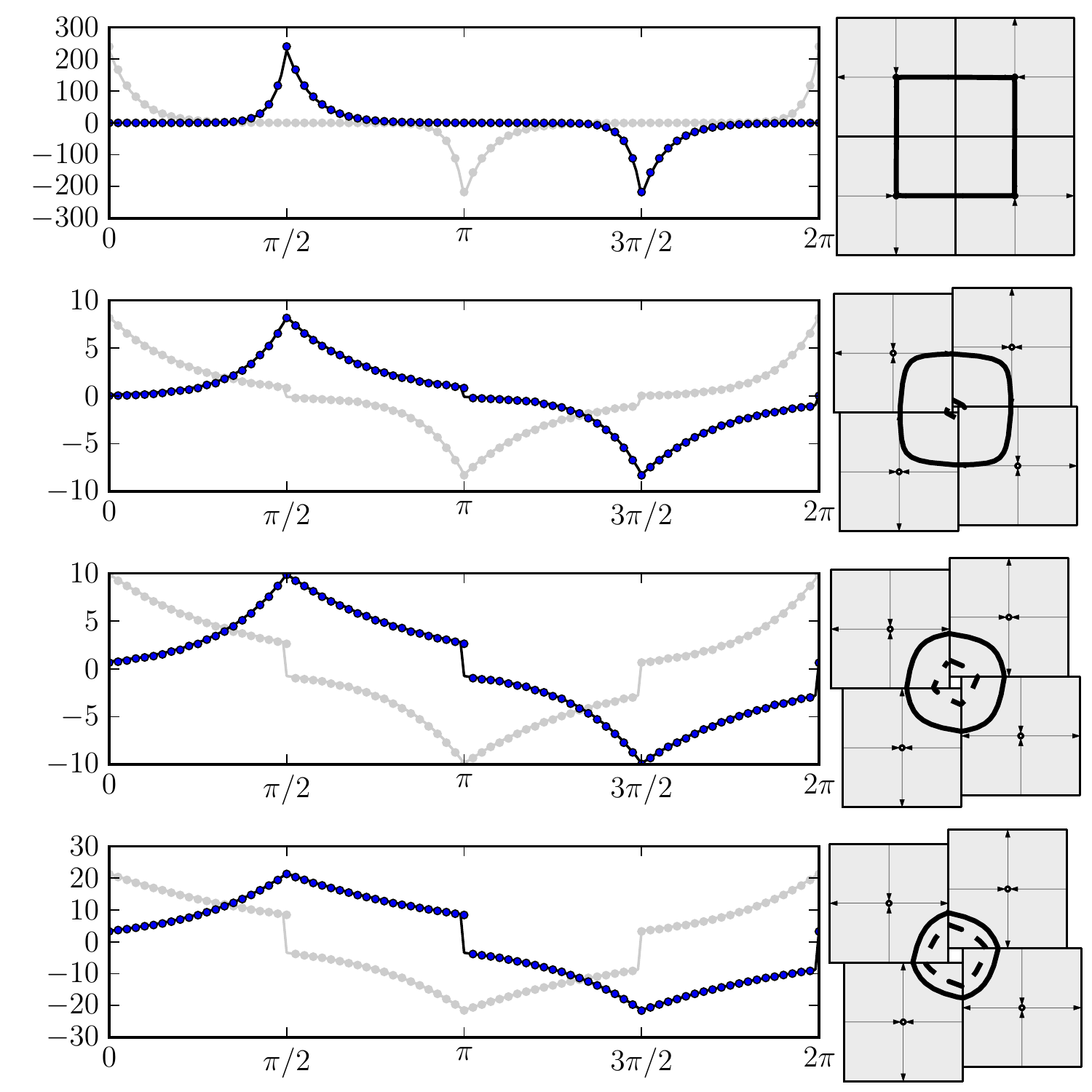}
\caption{Infinitesimal phase response curves (iPRCs) for the iris system with
various values of $a$.  As $a$ becomes small and the limit cycle approaches the
heteroclinic orbit, the phase response curve becomes dominated by peaks at the
edges between squares where flow changes from compressing trajectories outward
to expanding them inward.  As $a$ grows and the flow along the limit cycle
becomes more uniform, the iPRC becomes less sharply peaked.  
Line (dark blue): analytically obtained iPRC for a perturbation in the positive $x$ direction.  Points (dark blue): numerical
iPRC, calculated using an instantaneous
perturbation of $10^{-4}$ in the horizontal direction.  From top to bottom, $a=10^{-3}, 0.1, 0.2, 0.24$.
The light gray curve represents the corresponding iPRC for perturbations in the vertical (positive $y$) direction. The discontinuities
between the positive and negative portion of each curve are steps of height ($\ud+\sd$) (see \S \ref{sec:iPRC} for further details).}
\label{iris_prc_fig}
\end{figure}

\section{Asymptotic phase resetting behavior as $a\to 0$}
\label{sec:iris-PRC-asympt}

We wish to study the asymptotic behavior of the phase resetting curve at the heteroclinic limit, that is, as the parameter $a\to 0$.  This limit corresponds, at least by analogy, with the limit $\mu\to 0$ in the smooth vector field system of Equations \ref{eq:rotated_smooth_sho}.

\begin{lem}\label{lem:asymptotics1}
Under the conditions of Lemma~\ref{first_pass_lem}, the limit cycle entry value  scales as
$\ud=a+o(a)$, as $a\to 0$. 
\end{lem}
\begin{proof}
Recall that  $\ud$ is the smaller root of $\rho(u)$, so $\ud=a+(\ud)^\lambda$.  Since $\lambda>1$, clearly 
$\ud=a+o(\ud)$, as $\ud\to 0$, and $\ud>a$ whenever $\ud>0$.  For sufficiently small $a>0$ we have  $\ud$ as an implicitly defined function of $a$.  It is straightforward to see that $d\ud/da=1$ when $\ud=0$.  We wish to show that $\ud=a+o(a)$ as well, \textit{i.e.}~that $\lim_{a\to 0}(\ud-a)/a=0$.  Since $\ud=a+o(a)$, for any $\epsilon>0$ we can find a $\delta>0$ such that $1-a/\ud<\mbox{min}[\epsilon/2,1/2]$ whenever $0<\ud<\delta$. If $a<\delta$ then $(\ud-a)/a=1/(1-(1-a/\ud))-1=1/(1-x)-1=x/(1-x)$ where $x=1-a/\ud$.  Since $x<\mbox{min}[\epsilon/2,1/2]$, we are guaranteed that $x/(1-x)<\epsilon$, as required. \end{proof}

As $a$ decreases, the peaks of the iPRC grow, while their width shrinks.  The integral under each strictly positive region is 
\begin{equation}
V=\int_0^2 \beta_x(\varphi)\,d\varphi = \frac{1+\lambda-\lambda\ud-(\ud)^\lambda}{\lambda\log(1/\ud)}
\end{equation}
which goes  to zero as $\ud\to 0$ or equivalently as $a\to 0$.  Lemma \ref{lem:alpha-asympt} characterizes the behavior of the ``normalized'' iPRC $\alpha(\varphi)=\beta(\varphi)/V$ as $a\to 0$.
\begin{lem} 
\label{lem:alpha-asympt}
For the iris system the vector $\alpha(\varphi)=\beta(\varphi)/V$ has the following properties
as $a\to 0$:
\begin{enumerate}
\item For $\varphi=1$ or $3$, $\alpha_x\to\pm\infty$, respectively. For $\varphi\not\in\{1,3\}$, $\alpha_x\to 0$.
\item For $\varphi=0$ or $2$, $\alpha_y\to\pm\infty$, respectively. For $\varphi\not\in\{0,2\}$, $\alpha_y\to 0$.
\item  $\int_0^4||\alpha(\varphi)||\,d\varphi=4$ (in the $L_1$ norm).
\end{enumerate}
\end{lem}
\begin{proof}
As $a\to 0$ both $V\to 0$ and $\ud\to 0$.  
At $\varphi=1$, $\beta_x\equiv 1$ for all $a>0$, so as $a\to 0$ the ratio $\beta_x(1)/V\to+\infty$. On the other hand, for fixed $\varphi\in[0,2)\backslash\{1\}$, $(\ud)^{\lambda(1-\varphi)}\to 0$ as $\ud\to 0$.  
The behavior of $\alpha_x$ or $\alpha_y$ within the other intervals follows from the symmetry relations discussed in the preceding paragraph.  Thus 1 and 2 are established. For 3, note that by construction each interval of length two centered on a peak of one component makes a unit contribution to the integral.
That is,
$$1=\int_{[0,2)} \alpha_x(\varphi)\,d\varphi = -\int_{[2,4)} \alpha_x(\varphi)\,d\varphi = - \int_{[1,3)} \alpha_y(\varphi)\,d\varphi= \int_{[3,4)\cup[0,1)} \alpha_y(\varphi)\,d\varphi.$$
The integral of the $L_1$ norm $||\alpha(\phi)||$ is the sum of these terms.
\end{proof}
\begin{rem}
Lemma \ref{lem:alpha-asympt} and the symmetry of $\alpha$ together imply that each component of $\alpha$ converges weakly to a sum of delta function distributions on the circle:
\begin{eqnarray}
\alpha_x(\varphi)&\to&\delta(\varphi-1)-\delta(\varphi-3)\\
\alpha_y(\varphi)&\to&\delta(\varphi)-\delta(\varphi-2),
\end{eqnarray}
as $a\to 0$.
\end{rem}


Rewriting the full PRC in terms of $\alpha$ gives
\begin{equation}
Z(\eta,\varphi,a)=\frac{\eta\cdot\alpha(\varphi)\,V}{\log(1/\ud)(\ud-\lambda\sd)}
=
\left(\eta\cdot\alpha(\varphi)\right)M(\ud)
\end{equation}
where $M(\ud)=\frac{1+\lambda-\lambda\ud-(\ud)^\lambda}{\lambda\left(\log(1/\ud)\right)^2(\ud-\lambda(\ud)^\lambda)}$.  The function $M$ represents the magnitude of the PRC and diverges as $\left(\ud\log^2(\ud)\right)^{-1}$ as $\ud\to 0$.

The asymptotic phase response to an arbitrary instantaneous perturbation in direction $\eta$ is a sum of the response
to the perturbation component along the stable eigenvector direction, $\eta_s$ and along the unstable eigenvector direction, $\eta_u$, within whichever square the perturbation occurs.  The asymptotic behavior of the phase response in the heteroclinic limit for each component is distinct.  We rewrite the PRC to emphasize these contributions thus (restricted, for convenience, to the SW square, or $\varphi\in[0,1)$):
\begin{equation}
Z(\eta,\varphi,a)=\frac{\eta_s\left((\sd)^{(1-\varphi)}\right)+\eta_u\left((\ud)^\varphi\right)}
{\log(1/\ud)(\ud-\lambda\sd)}.
\end{equation}

First consider the response to perturbation along the unstable direction.  Recall that for any $p>0$,
$u^p\log u\to 0$ as $u\to 0^+$.  It follows after a brief calculation that for any $\varphi\in[0,1)$
\[
\lim_{u\to 0^+}\left|\frac{u^\varphi}{\log(1/u)(u-\lambda u^\lambda)}\right|=+\infty.
\]
This result means that regardless of the phase of the perturbation, the sensitivity of the asymptotic phase to small displacements parallel to the unstable eigenvector diverges as the system approaches the heteroclinic bifurcation.  The result is intuitively appealing because when the system is close to the heteroclinic bifurcation the long period leads to a larger ``compounding'' effect enhancing the cumulative result of a small perturbation within a given square.  However, as is often the case when dealing with nested limits, intuition can be misleading.
In contrast to the preceding situation, consider the response to perturbation along the stable direction.  The behavior of the limit now depends on the parameter $\lambda$:
\[
\lim_{u\to 0^+}\left|\frac{u^{\lambda(1-\varphi)}}{\log(1/u)(u-\lambda u^\lambda)}\right|=\left\{\begin{array}{ll}0,& \varphi\in[0,1-1/\lambda] \\ +\infty,& \varphi\in(1-1/\lambda,1) \end{array} \right..
\]
Thus, in the heteroclinic limit, small perturbations parallel to the stable eigenvector become inconsequential if they occur at an early enough phase.  At a particular phase,
\begin{equation}
\varphi_*=1-\frac{1}{\lambda}, 
\end{equation}
the response becomes highly sensitive to arbitrarily small perturbation. 

\begin{rem}This concludes the proof of Theorem \ref{thm:main}, Part 3.\end{rem}

For comparison we can obtain from Equations \ref{eq:closest} the fractional phases at which the trajectory is closest to the saddle point (in the first square) or moving at the slowest speed:
\[
\varphi_{\mbox{closest}}=\frac{1-\log\lambda/(2\log \ud)}{\lambda+1}, 
\hspace{1in}
\varphi_{\mbox{slowest}}=\frac{1-3\log\lambda/(2\log \ud)}{\lambda+1}. 
\]
As $\ud\to 0$ both phases converge to $\varphi_0=1/(\lambda+1)$, which is distinct from the phase at which the PRC's asymptotic sensitivity to perturbations in the stable direction changes.  Curiously, the asymptotic value $\varphi_0$ of the phases of the slowest and closest point along the trajectory coincides with $\varphi_*$, the critical phase for sensitivity, precisely when $\lambda$ is equal to the golden ratio, $(1+\sqrt{5})/2\approx 1.618\cdots$, because under those conditions $1 - 1/\lambda = 1/(\lambda + 1)$.

\section{Isochrons}\label{sec:isochrons}

Provided $a>0$ in the iris system, we may define isochrons as the level sets of the asymptotic phase function $\theta(x)$, for any point $x$ in the basin of attraction of the limit cycle.  
As described in \cite{BrownMoehlisHolmes:2004:NeComp,Guckenheimer+Holmes1990,Izhikevich2007} and elsewhere, the points on a given isochron will converge over time
to a single point on the limit cycle with a particular phase.  Thus, in a 
sense, they represent points with the same asymptotic ``time".  
As discussed above, given a limit cycle 
 $\{\gamma(t)\}_{t=0}^{T_a}$
the asymptotic phase of a point $x_0$ is defined as the unique value in $\theta\in[0,\theta_{\mbox{max}})$ such  that the limit
\begin{equation}
\lim_{t\to\infty}|x(t)-\gamma(t+\theta(x_0))|=0,
\end{equation}
where $x(0)=x_0$ is the initial condition for the trajectory. 
The usual approach for finding isochrons is based on 
an adjoint equation method, \textit{cf.}~(\cite{BrownMoehlisHolmes:2004:NeComp}, Appendix A; see also \cite{ErmentroutTerman2010book}, Chapter 11).
  From the chain rule, the phase field must satisfy
\begin{equation}\label{eq:adjoint}
\frac{d}{dt}\theta(x(t))=\left[\left(\vec{\nabla}\theta\right)\left(x(t)\right)\right]\cdot\frac{dx}{dt} = \frac{\theta_{\mbox{max}}}{T(a)}
\end{equation}
where $T(a)$ is the period of the limit cycle for a given value of $a>0$.    
Let $(0\le s \le 1, 0\le u \le 1)$ be local coordinates relative to any saddle
of the iris system, \textit{i.e.}~such that the flow satisfies $\dot{s}=-\lambda s, \dot{u}=u$.  Writing $\theta_v$ for $\partial\theta/\partial v$, Equation \ref{eq:adjoint} reads
 \begin{equation}
-\lambda s\theta_s + u \theta_u = \theta_{\mbox{max}}/T(a).
\end{equation}
By virtue of the fourfold rotational symmetry of the iris system, the following boundary conditions are required in addition to the PDE \ref{eq:adjoint}:  for $0<s<1-a$, 
\begin{equation}
\theta(s,1)=\theta(1,s+a)+1
\end{equation}
The boundary condition renders the PDE nonlinear, and a simple closed form solution is not available.  Solving the PDE numerically, however, is equivalent to finding the isochrons \textit{via} direct simulation of trajectories for a suitable grid of initial conditions.  We used a 400$\times$400 square mesh of initial conditions within each square, tracking solutions until either they left the system of large squares (initial conditions on the interior of the unstable limit cycle) or converged with a suitable tolerance to a small neighborhood of the limit cycle, at which point the asymptotic phase could be assigned.  Figure \ref{iris_isochron_fig} shows an example of the isochrons generated with this technique. 
\begin{figure}[htpb] \centering
\includegraphics[width=.9\figwidth]{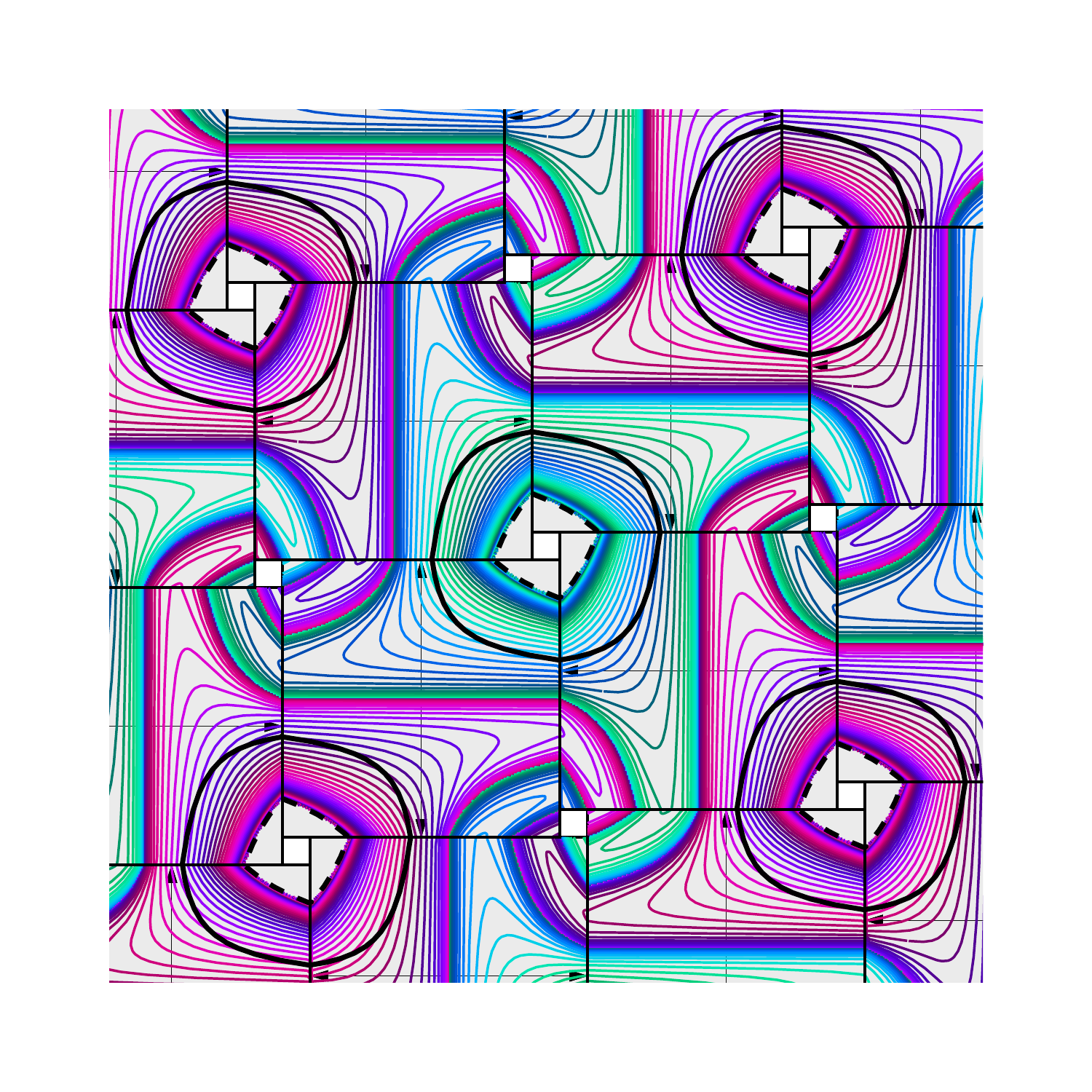}
\caption{Isochrons for the iris system obtained \textit{via} numerical iteration of the map from initial conditions to a sequence of boundary crossing.  
The iris system on the torus has two distinct stable limit cycles for $a>0$, illustrated here in two color schemes (green-blue and fuschia-violet).  The tiling of the plane by periodic extension of the basic torus system is shown.  Solid lines indicate the stable limit cycles, dashed lines indicate the unstable limit cycles.  All limit cycles rotate clockwise.  The flow inside the small squares is not defined, and points inside the unstable limit cycles are not part of the basin of attraction of the stable limit cycles.  The stable manifolds of each saddle comprise the separatrices between the basins of attraction of distinct stable limit cycles.  
The system shown has $a=.2$ and $\lambda=2$.  Colors (online) indicate the phase.
Note the compression of the level curves near the boundaries between adjacent squares, which correspond to peaks in the phase response curve.
Note also that the isochrons are \emph{not} strictly parallel to the stable eigenvalue direction (\textit{c.f.}~the discussion in \S \ref{ssec:homoclinic}).  Please see corresponding movie (file: \texttt{iris\_isochrons.mpg}); this   animation illustrates the flow together with the isochron structure, by evolving points of equal phase forward together in time. After five seconds the animation changes to showing isochron structures for different values of the unstable/stable manifold offset parameter $a$, from $a=0$ to $a=0.247$ and back again.
}
\label{iris_isochron_fig}
\end{figure}

\section{Smooth System}
\label{sec:smooth}

In order to compare the family of infinitesimal phase response curves obtained
for the iris system with that of a family of continuous vector fields
experiencing a similar heteroclinic bifurcation with fourfold symmetry, we
numerically evaluated the phase response for small instantaneous perturbations in the
horizontal and vertical directions for limit cycles given by the system
\ref{eq:smooth_sho} for positive values of the twist parameter $\mu$.  Figure
\ref{sine_prc_fig} shows good qualitative agreement with the concentration of
phase response sensitivity at points along the trajectory near integer values
of the phase $\varphi$, \textit{cf.}~Figure \ref{iris_prc_fig}.

\begin{figure}[htpb] \centering
\includegraphics[width=\figwidth]{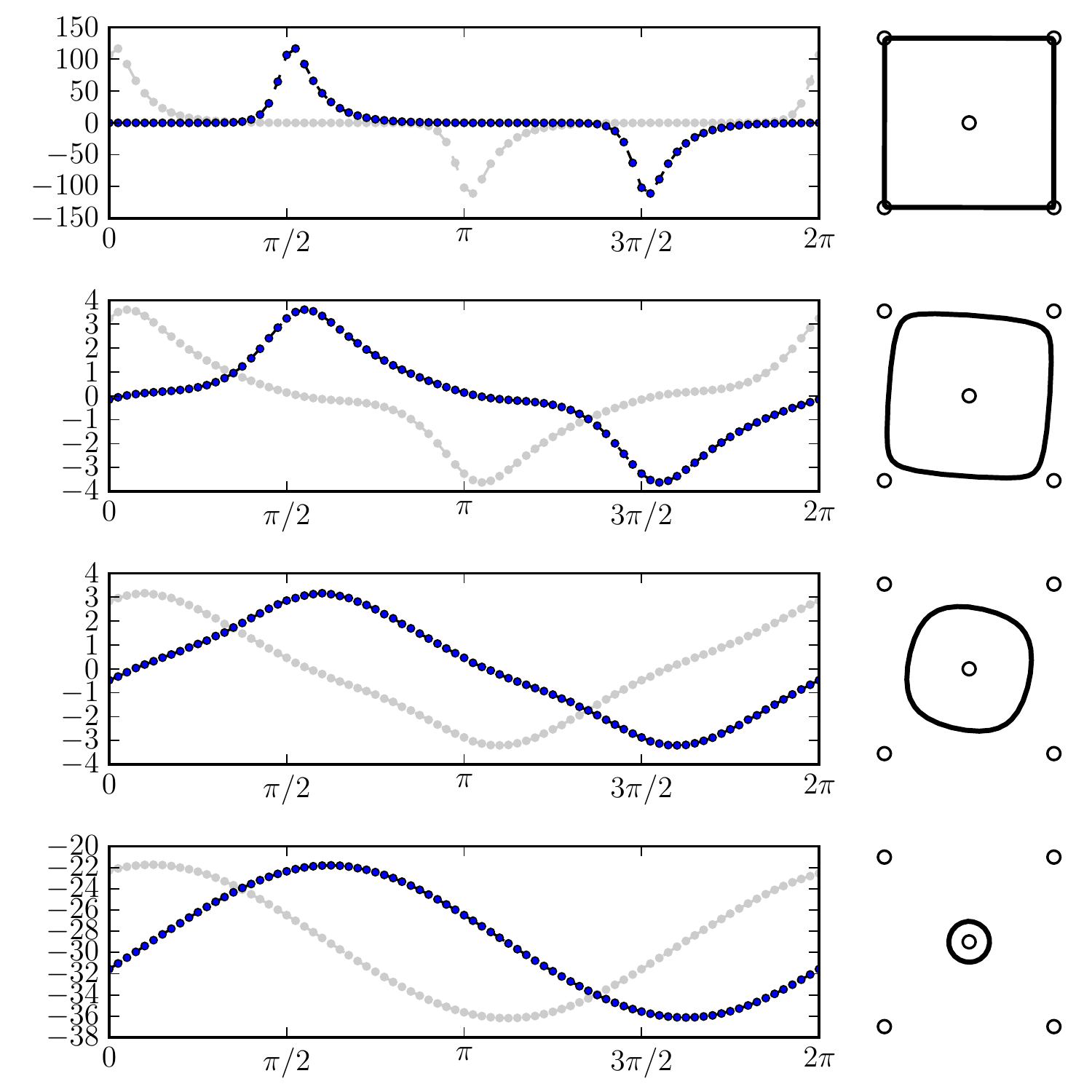}
\caption{Infinitesimal phase response curves for the smooth system with
various values of $\mu$.  As $\mu$ becomes small and the limit cycle approaches 
the heteroclinic orbit, the phase response curve becomes dominated by peaks 
much like the iris system shown in Figure~\ref{iris_prc_fig}.  As $\mu$
grows, the phase response curve becomes a sinusoidal as the system approaches 
an Andronov--Hopf bifurcation.  Points connected by dashed line (dark blue): numerical
infinitesimal phase response curve, calculated using an instantaneous
perturbation of $10^{-4}$ in the horizontal direction. 
Dots (dark blue): infinitesimal phase response curve for perturbations in the 
horizontal direction.
 Light gray: infinitesimal phase response curve for perturbations in the vertical direction.  
  From top to bottom, $\mu= 10^{-3}, 0.1, 0.3, 0.45$.}
\label{sine_prc_fig}
\end{figure}

\section{Discussion}

The iris system introduced here is one of only a handful of nonlinear dynamical systems for which an explicit form for the limit cycle or the phase response curve is known.  For this system we show analytically that the sensitivity of the response to small instantaneous perturbations depends (as one would expect) on the distance to the heteroclinic bifurcation, on the phase at which the perturbation occurs, and the direction of the perturbation in the phase plane.  In addition, we find the intuitively appealing result that regardless of phase, as the static perturbation away from the heteroclinic vector field diminishes, the phase response as measured by the iPRC becomes hypersensitive to any small perturbation parallel to the unstable eigenvector direction, in that the iPRC diverges as the bifurcation parameter of the iris system $a\to 0$.  Unexpectedly, however, we also find that for perturbations along the \emph{stable} eigenvector direction, the response to some will diverge while the response to others will have no effect in the limit as $a\to 0$.  The system appears ultrasensitive in this case to the phase at which the perturbation occurs, with perturbations sufficiently far in advance of the approach towards the saddle point effectively absorbed by the flow, in contrast to perturbations beyond a critical phase, $\varphi_*=1-1/\lambda$, which diverge as $a\to 0$.  From a biological point of view, this dual sensitivity to the timing and direction of the transient perturbation is significant if flows structured by fixed points are to be exploited in nature as control points for rhythmic behaviors.  In neural systems, for instance, it is commonplace to consider perturbations restricted to a single dimension in a multidimensional flow, namely perturbations along the dimension of membrane potential \cite{BrownMoehlisHolmes:2004:NeComp}.  When a neuron's membrane potential lies in the linear regime, its dynamics is typically dissipative due to membrane conductances, tending to align the voltage direction with the stable direction for a subthreshold fixed point, at least for some portion of the flow. In general, if slowly adaptive modulatory processes within an organism's control circuitry can impose quasistatic perturbations of an existing vector field to move trajectories closer to or farther away from a homo- or heteroclinic bifurcation point, the modulatory processes can filter which perturbations the system becomes sensitive to at different phases of the ongoing oscillation.  Investigating these possibilities in specific systems such as the feeding central pattern generator of the marine mollusk \textit{Aplysia} will make it possible to test this prediction empirically.


\subsection{Sensitivity and control}

Our results suggests that for systems in which a control parameter can change the approach of trajectories towards a hyperbolic saddle fixed point with a one dimensional unstable manifold, the sensitivity to specific kinds of perturbations could be actively managed by manipulating the critical phase $\varphi_*$. This is a novel kind of control for rhythmic behaviors; additional mechanisms for control include manipulating the period of a typical trajectory by controlling the time spent near a saddle point, and regulating the variable dwell time in different states for a limit cycle trajectory passing near multiple unstable fixed points.  

Oscillations arising from co-dimension one bifurcations other than a saddle node-homoclinic have phase response curves near the bifurcation point that are smooth and sinusoidal.  By contrast, many phase response curves observed in real systems show sharp transitions between small and large responses, even in systems generating  highly regular behaviors such as swimming in the lamprey. Lamprey motor units participating in the CPG underlying smooth swimming showed broad regions of near zero phase response combined with steep changes in sensitivity in other regions, going from zero to peak phase response in an interval corresponding to about 10\% of the limit cycle duration (\textit{c.f.}~Figure 6 of \cite{VarkonyiKiemelHoffmanCohenHolmes:2008:JCNS}).  Preliminary data obtained from the marine mollusk \textit{Aplysia californica} during feeding behavior show extended, variable dwell times in preferred regions of phase space separated by rapid transitions from region to region \cite{Shaw-Chiel-Thomas-2010-SFN-poster}, suggesting
the possibility that this central pattern generator's dynamics may resemble that of a limit cycle making close encounters with multiple unstable fixed points.  

The infinitesimal PRC is equal to the gradient of the asymptotic phase function evaluated at the limit cycle.  For the iris system, the level curves appear to pinch together at the boundaries where they abruptly change direction, and this pinching means the density of the level curves is highest at the square boundaries.  
One might naively expect that the phase response should be 
most sensitive to perturbations  immediately adjacent to the saddle points; instead the peak sensitivity occurs at boundaries between regions dominated by one or another saddle.  This effect is qualitatively present both in the piecewise linear iris system and in the smooth system considered here.  Within the piecewise linear regions of the iris system the flow is homogeneous; it is the boundaries that make the system nonlinear and it should not be surprising that that is where the greatest sensitivity occurs.  In the sine system, it appears that near the saddles the flow is governed by the local linear approximation, while in the region between saddles the flow is ``more nonlinear'' coinciding with greater  sensitivity to perturbation.  


\subsection{Comparison to the PRC near a homoclinic bifurcation}
\label{ssec:homoclinic}

While the structure and asymptotic behavior of the phase response curve described in Theorem \ref{thm:main} were derived specifically for the iris system, we expect the qualitative picture to hold for generic one-parameter families of limit cycles verging on a heteroclinic or a homoclinic bifurcation of $C^1$  vector fields as well.  Treating the general case lies beyond the scope of this paper; however, it will be useful to compare our results with the analysis of phase resetting near a homoclinic bifurcation given in (\cite{BrownMoehlisHolmes:2004:NeComp}, \S 3.1.3).  

In \BMH the authors consider a vector field on $\R^n$ with a hyperbolic saddle fixed point at the origin with a single unstable eigenvalue $\lambda_u$ and unstable eigenvalues $\lambda_{s,j}$ with $\lambda_u<\lambda_s=\min_{j}[\lambda_{s,j}]$.  For positive values of a bifurcation parameter $\mu$, the system is assumed to have a limit cycle that spends the overwhelming majority of its time in a box $B=[0,\Delta]^n$.  Trajectories exit the box when the coordinate along the unstable eigendirection $x=\Delta$ and are instantaneously reinjected with $x=\epsilon$.  When $n=2$ the geometry is similar to that of the iris system (compare Figure 3 of \BMH with our Figure \ref{fig:sine_to_iris_fig_iris}).
The main difference is that in the homoclinic system, the value of the injection point along the unstable limit coordinate is assumed to be independent of the egress point, whereas in the iris system its dependence is specified by the map $h=f_e\circ f_l$ (Figure \ref{iris_return_map_fig}).  Equivalently, it is as if the derivative of the function $h$  were zero at the fixed point $\ud$ rather than having finite slope.  This situation  would occur for the homoclinic system if the compression of the flow taking trajectories from the egress boundary to the ingress boundary were sufficiently great.   However, in analyzing the homoclinic system, one must also assume that the time spent along the portion of the limit cycle outside the box $B$ is vanishingly small.  Reconciling these two limiting processes for the general homoclinic case remains an interesting problem.  

For purposes of comparison, note that $\epsilon/\Delta$ in \BMH corresponds to $\ud$ for the iris system, and the phase $\theta\in[0,2\pi)$ in \BMH corresponds to $2\pi\varphi$ in the iris system.
Table \ref{tab:homoclinic} compares the infinitesimal PRCs for the two systems.  
In \BMH the perturbation is assumed to be in a particular direction corresponding to a voltage deflection in a neural model; the factor $\nu_x$ arises in the change of coordinates from the direction of voltage peturbations relative to the unstable eigenvector direction.  The flow is assumed to be infinitely compressive during the time between egress and reinjection, and consequently it is assumed that only components of the perturbation along the unstable eigendirection contribute to the iPRC.  This assumption holds for the iris system in the limit as the unstable-to-stable eigenvalue ratio $\lambda_u/\lambda_s=\lambda\to\infty$, but not for finite values of $\lambda$. The $\lambda<\infty$ case includes additional corrections; see  Table \ref{tab:homoclinic}.  

Finally, assuming that the phase response is independent of displacement along the stable eigenvector direction is equivalent to assuming that the isochrons run parallel to the stable eigenvector throughout the box $B$.  For the iris system the slopes of the isochrons do approach zero as $\lambda\to \infty$, becoming more and more parallel to the stable eigenvector, but for $\lambda=\infty$ the asymptotic phase and the isochrons are no longer well defined.  For finite values of $\lambda$ the isochrons are well defined, but are not parallel to the stable eigenvector.  Compare the isochrons in the SW quadrant of the iris system shown in Figure \ref{iris_isochron_fig}.

\begin{table}[tbp]
   \centering
   \renewcommand{\arraystretch}{2}
   \begin{tabular}{ccc}\hline
    System &$Z_u(\varphi)$ & $Z_s(\varphi)$\\ \hline
   Homoclinic & $\frac{(\nu_x/\Delta)u^{\varphi}}{u\log(1/u)}$ & 0\\
   Iris & $\frac{u^\varphi}{(u-\lambda s)\log(1/u)}$& $\frac{s^{(1-\varphi)}}{(u-\lambda s)\log(1/u)}$ \\
   \hline
   \end{tabular}
   \caption{Phase response curve for homoclinic system \BMH  and the iris system, compared.  For clarity, we write $u, 0<u<1,$ for $\ud=\epsilon/\Delta$, and the phase $0\le\varphi\le 1$ for both systems. Here $s=u^{\lambda}$;   in the $\lambda\to\infty$ limit $s\to 0$, and the systems coincide (up to a scaling factor $\nu_x/\Delta$ related to the choice of perturbation direction).  The iPRC for perturbations parallel to the unstable (resp.~stable) eigenvector direction is given by $Z_u$ (resp.~$Z_s$).} \label{tab:homoclinic}
\end{table}

\subsection{Phase resetting in the absence of an asymptotic phase}

Deterministic limit cycles are not the only way to represent (approximately) periodic biological rhythms.  Alternatives include heteroclinic or homoclinic attractors perturbed by small amplitude noise \cite{Bakhtin2010ProbThyRelatFields,StoneHolmes1990SIAMJApplMath} as well as fixed points of spiral sink type subject to small noisy perturbations \cite{BolandGallaMcKane2008JStatMech,Izhikevich2007}.
A system comprising a determinstic flow with a stable spiral fixed point, when perturbed by small to modest amounts of noise, will show noisy oscillatory trajectories that can be difficult to
distinguish from a small amplitude limit cycle \cite{BolandGallaMcKane2008JStatMech}.  For example,
the small oscillations in membrane potential of a nerve cell brought near to firing by a steady depolarizing current of modest size may be due either to ``noisy spiral sink" dynamics or to ``noisy small limit cycle" dynamics \cite{StiefelFellousThomasSejnowski2010EJN}.  The classical definition of asymptotic phase breaks down for noisy systems, for spiral sinks, and for heteroclinic or homoclinic orbits.  Nevertheless ``phase resetting" is actively studied in such systems both experimentally \cite{ErmentroutSaunders2006JCNS,ErmentroutGalanUrban:2008:TNsci,StiefelFellousThomasSejnowski2010EJN,StiefelGutkinSejnowski:2008:PLoSOne} and theoretically \cite{IchinoseAiharaJudd:1998:IntJBifChaos,RabinovitchRogachevskii:1999:Chaos,RabinovitchThiebergerFriedman:1994:PRE}.  For the Bonhoeffer--van der Pol equations, for example, Rabinovitch and colleagues showed how to define an analog of the classical isochron, for point converging to a spiral sink along a distinguished trajectory (a ``T-attractor") \cite{RabinovitchRogachevskii:1999:Chaos}.
Generally speaking, noise effects can expose the structure of bifurcations in neural systems through stochastic resonance \cite{GaiDoironRinzel:2010:PLoSComputBiol} or spike time bifurcations \cite{neurocomputing:ThomasEtAl:2003,ToupsEtAl2011NC-toappear}, and it would be useful to have a broader
theory of ``phase response'' going beyond the case of deterministic limit cycles.

\section{Acknowledgments}


This work was supported by National Science Foundation grants DMS-0720142     
and DMS-1010434.  The authors thank C.~Turc and W.~Guo for helpful comments concerning solutions of the adjoint equation, and R.~Galan and A.~Horchler for helpful discussion.

\bibliographystyle{siam} 
\bibliography{math,neuroscience,PJT,Respiration,Dicty}
\end{document}